\documentclass[12pt,reqno]{amsart}
\usepackage{amsmath,amssymb,amsthm,amscd}
\usepackage{mathrsfs}
\usepackage[old]{old-arrows} 
\usepackage{stmaryrd}

\usepackage[utf8]{inputenc}
\usepackage[english]{babel}

\usepackage{tikz}
\usepackage{tikz-cd}
\usetikzlibrary{babel} 
\usetikzlibrary{cd} 
\tikzset{
  symbol/.style={
    draw=none,
    every to/.append style={
      edge node={node [sloped, allow upside down,
      auto=false]{$#1$}}}
  }
}

\usepackage{graphicx}
\usepackage[alphabetic]{amsrefs}
\usepackage[bookmarks=false,hypertexnames=false,colorlinks, citecolor=blue]{hyperref}
\usepackage{bookmark}
\usepackage{verbatim,color,geometry}
\usepackage{marginnote}
\usepackage[colorinlistoftodos, textsize=small]{todonotes}
\makeatletter 
\@mparswitchfalse%
\makeatother
\reversemarginpar

\usepackage[shortlabels]{enumitem} 

\newtheorem{thm}{Theorem}[section]
\newtheorem{prop}[thm]{Proposition}
\newtheorem{defn-thm}[thm]{Theorem-Definition}
\newtheorem{defn-lem}[thm]{Lemma-Definition}
\newtheorem{cor}[thm]{Corollary}

\newtheorem*{thm*}{Theorem}
\newtheorem*{prop*}{Proposition}
\newtheorem*{cor*}{Corollary}
\newtheorem*{lem*}{Lemma}

\theoremstyle{definition}

\newtheorem{rem}[thm]{Remark}

\newtheorem*{claim*}{Claim}
\newtheorem*{rem*}{Remark}

\newcommand{\su}{\subset}

\newcommand{\bx}{{\bar{x}}}
\newcommand{\by}{{\bar{y}}}
\newcommand{\bz}{{\bar{z}}}

\newcommand{\ba}{{\bar{a}}}
\newcommand{\bbb}{{\bar{b}}}
\newcommand{\bc}{{\bar{c}}}
\newcommand{\y}{y}
\newcommand{\yy}{y}

\newcommand{\ttto}{\dashrightarrow} 
\newcommand{\tto}{\longrightarrow}

\newcommand{\wt}{\widetilde}

\newcommand{\pa}{\partial}

\newcommand{\ov}{\overline}

\newcommand{\isom}{\cong}

\newcommand{\PP}{\mathbb{P}}

\newcommand{\AAA}{\mathbb{A}}
\newcommand{\F}{\mathbb{F}}

\newcommand{\OO}{\mathcal{O}}

\newcommand{\pp}{\mathfrak{p}}
\newcommand{\fq}{\mathfrak{q}}

\newcommand{\al}{\alpha}
\newcommand{\be}{\beta}

\newcommand{\ga}{\gamma}

\newcommand{\de}{\delta}
\newcommand{\De}{\Delta}
\newcommand{\etta}{\eta}
\newcommand{\sig}{\sigma}

\newcommand{\ka}{\kappa}

\newcommand{\eps}{\varepsilon}
\newcommand{\vf}{\varphi}

\DeclareMathOperator{\Spec}{Spec}

\begin{document}

\title[A class of geometrically elliptic fibrations]
{A class of geometrically elliptic fibrations \\
by plane projective quartic curves}


\author{Cesar Hilario}
\address{Mathematisches Institut, Heinrich-Heine-Universit\"at, 40204 D\"usseldorf, Germany}
\email{cesar.hilario.pm@gmail.com}

\author{Karl-Otto Stöhr}
\address{IMPA, Estrada Dona Castorina 110, 22460-320 Rio de Janeiro, Brazil}
\email{stohr@impa.br}

\subjclass[2010]{14G17, 14G05, 14H05, 14H45, 14D06, 14E05}


\dedicatory{October 28, 2025}


\begin{abstract}
We investigate fibrations by non-hyperelliptic curves of arithmetic genus three and geometric genus one in characteristic two. 
Assuming that there is only one moving singularity and that its image in the Frobenius pullback of the fibration has degree one over the base, 
we provide a complete classification up to birational equivalence. This relies on an in-depth analysis of the generic fibres, whose geometry we describe explicitly. 
We prove that these fibrations are covered by elliptic fibrations, and that the covers are birational on the fibres but purely inseparable of exponent one on the bases.
\end{abstract}

\maketitle

\setcounter{tocdepth}{1}

\section{Introduction}

Bertini's theorem on moving singularities, published in the 1880s, also known as the Bertini-Sard theorem, has become a fundamental tool in Algebraic Geometry.
It states that in characteristic zero almost every fibre of a dominant morphism $\Phi: T \to B$ of smooth algebraic varieties over an algebraically closed field $k$ is  smooth.
However, as noticed by Zariski in the 1940s, the theorem may fail in positive characteristic, that is, though the total space $T$ is smooth, the fibration $\Phi:T\to B$ may admit moving singularities.
A moving singularity is a horizontal prime divisor on the total space $T$ with the property that each of its points is a singular point of the fibre to which it belongs. 

By a theorem of Tate \cite{Tate52,Sc09}, moving singularities of fibrations by integral curves can only exist if the characteristic is a prime $p$ not larger than $2(g-\ov g)+1$, where $g$, resp. $\ov g$, denotes the arithmetic genus, resp. geometric genus, of the general fibre. 
More generally, assuming that the morphism $\Phi:T\to B$ is proper, 
by Rosenlicht's genus drop formula \cite[Formula~2.3]{HiSt22} the number of the moving singularities is equal to $g-\ov g$, provided they are counted according to their $\delta$-invariants, 
which are multiples of $\tfrac{p-1}{2}$ by \cite[Proposition~2.5]{HiSt22}.

The most prominent counter-examples to Bertini's theorem are the quasi-elliptic fibrations, which are the fibrations by plane integral curves of degree $d=3$ with a moving cusp, and so 
$g=(d{-}1)(d{-}2)/2=1$, $\delta=1$, $\ov g =g-\delta=0$ and $p=2$ or $p=3$. 
They arose in the seminal work of Bombieri and Mumford on the extension of Enriques' classification of algebraic surfaces to positive characteristic \cite{BM76}. 

Fibrations by plane integral curves of degree $d=4$ (i.e., $g=3$) that admit moving singularities 
were investigated in characteristic $p=3,5\text{ and }7$ in the papers \cite{Sal11}, \cite{BMSa25}, \cite{St07} and \cite{St04}. 
A birational classification of the case $d=4$, $p=2$ and $\ov g=0$ has been recently established (see \cite{HiSt24}).

In the present paper we assume that $d=4$, $p=2$ and $\ov g=1$, that is, we investigate the geometry of the fibrations by plane integral quartic curves of geometric genus one in characteristic two. We assume in addition that there is only one moving singularity and that its image in the normalization of the Frobenius pullback of the fibration $\Phi:T\to B$ has degree one over the base $B$. 

To get a birational classification of these fibrations we work in the arithmetic setting of function field theory. 
If the ground field $k$ is replaced by the function field $K:=k(B)$ of the base, then the higher dimensional function field $k(T)|k$ of the total space $T$ becomes the one-dimensional function field $k(T)|k(B)=K(C)|K$ of the generic fibre 
\[
C := T \times_B \Spec(K),
\] 
which is an integral proper curve over $\Spec(K)$. 
The closed points of $C$ correspond bijectively to the horizontal prime divisors of the fibration $\Phi$, i.e., to the prime divisors of the total space $T$ that cover the base $B$. 
As the total space $T$ is smooth, the generic fibre $C$ is regular, 
and so the local rings of its closed points are the discrete rank one valuation rings of its function field $K(C)|K$. 
Though the generic fibre $C|K$ is regular, it admits a non-smooth point, 
namely the closed point that corresponds to the moving singularity. 
Our main result (see Section~\ref{2025_10_17_19:35}) gives an explicit description of the generic fibre $C|K$.
\begin{thm}
Let $C$ be a regular proper geometrically integral curve of arithmetic genus $g=3$ and geometric genus $\ov g = 1$ over a field $K$ of characteristic $p=2$.
Assume that the curve has only one non-smooth point $\pp$, and that its image $\pp_1$ in the normalization $C_1|K$ of its (first) Frobenius pullback is a rational point. Then the following holds.
\begin{enumerate}[\upshape (a)]
\item 
The curve $C|K$ is hyperelliptic if and only if the $j$-invariant of the elliptic curve $C_1|K$ is equal to zero.

\item 
If the curve $C$ is non-hyperelliptic then the $j$-invariant $j$ of $C_1$ is a non-zero square, and $C$ is isomorphic to a plane projective quartic curve given by an equation
\[cx^4+(z^2+by^2+abx^2)(e(z^2 + b y^2 + a b x^2) + y^2 + xy + a x^2) = 0 \]
where $a,b,c,e \in K$, $b \notin K^2$ and $c^2e^2j=1$.
The converse is also true, i.e., there are no other conditions on the $j$-invariant and the coefficients $a,b,c,e$.
\end{enumerate}
\end{thm}
\noindent
To describe the isomorphism classes we decide when two quadruplets of coefficients determine isomorphic curves (see Corollary~\ref{2025_03_13_19:20}~\ref{2025_03_13_19:22}).

In Section~\ref{2025_10_01_15:20} we investigate the fibration $\Phi:T\to B$ that is uniquely determined up to birational equivalence by the function field 
$K(C)|K=k(T)|k(B)$ of the generic fibre $C|K$. The geometric generic fibre
\[
C \times_{\Spec(K)} \Spec(\ov K ) = T \times_B \Spec\ov{k(B)},
\]
which describes the behaviour of most special fibres, is a plane projective integral quartic curve over $\Spec(\ov K)$ with the remarkable property 
that its tangents (at the smooth closed points) meet in a common intersection point, 
and are bitangents or pass through the unique non-smooth point on the quartic curve.
In particular, the quartic does not admit ordinary inflection points. 
However, there are two non-ordinary inflection points, which are the points where the tangency points of the bitangents collide. 

We prove that the fibration $\Phi$ is covered by an elliptic fibration, a cover that is birational on the fibres but purely inseparable of exponent one on the bases. 
This implies that the geometric generic fibre (after excluding its generic point) is endowed with the structure of an abelian group whose neutral element is its unique non-smooth point. We show that its elements of order four are precisely its two (non-ordinary) inflection points.

\section{Geometrically elliptic function fields of genus three}
\label{2025_10_17_19:35}

By \cite[II.7.4]{EGA},
there is a one-to-one contravariant correspondence between the regular proper geometrically integral curves $C$ defined over a field $K$ and the one-dimensional separable function fields $F|K$, given by the assignment $C|K \mapsto F|K = K(C)|K$. 

Let $F|K$ be a one-dimensional separable function field in characteristic $p>0$, of genus $g=3$ and of geometric genus $\ov g=1$.
By Rosenlicht's genus drop formula 
the number of the singular primes of $F|K$ counted with their singularity degrees is equal to $g - \ov g = 2$. We assume that $F|K$ admits only one singular prime $\pp$, and so the singularity degree of $\pp$ is equal to $\de(\pp)=2$.
Since $\de(\pp)$ is a multiple of $(p-1)/2$ \cite[Proposition~2.5]{HiSt22}, the characteristic $p$ is not larger than $5$.
The cases where $p=3$ and $p=5$ were investigated in \cite{BMSa25} and \cite{St07}. We will assume that $p=2$.
We further assume that the singular prime $\pp$ is purely inseparable, i.e., its residue field $\ka(\pp)$ is a purely inseparable extension of the constant field $K$.
This means that the uniqueness of the singular prime $\pp$ is preserved under finite separable constant field extensions \cite[Corollary~2.17]{HiSt22}.
Then it follows from \cite[Theorem~2.24]{HiSt22} that the restriction $\pp_3$ of the prime $\pp$ to the third Frobenius pullback $F_3|K := KF^{p^3}|K$ is rational, i.e., its degree is equal to $\deg(\pp_3)=1$.
Thus the restriction $\pp_1$ of $\pp$ to the first Frobenius pullback $F_1|K := F^p K|K$ has degree one, two or four.
In this paper we assume  that $\deg(\pp_1) = 1$.
As $p=2$, $g=3$ and $\ov g = 1$, the genus of $F_1|K$ is equal to $g_1=1$ \cite[Corollary~2.7~(ii), (iii)]{HiSt22}.
Hence $F_1|K$ equipped with the rational prime $\pp_1$ is an elliptic function field.

We will apply Tate's treatment of elliptic curves, which works in any characteristic \cite{Tate74}.
As $g=1$ Riemann's theorem implies that
\[ \dim H^0(\pp_1^n) = n \quad \text{for each positive integer $n$.} \]
Thus there exist functions $x \in H^0(\pp_1^2) \setminus K$ and $y \in H^0(\pp_1^3) \setminus H^0(\pp_1^2)$. Hence $[F_1:K(x)] = 2$, $[F_1:K(y)] = 3$ and therefore
\[ F_1 = K(x,y). \]
Moreover
\begin{align*}
    H^0(\pp_1^{2n}) &= \bigoplus_{i=0}^n K x^i \oplus \bigoplus_{i=0}^{n-2} K x^i y \quad \text{for each $n\geq1$,} \\
    H^0(\pp_1^{2n+1}) &= \bigoplus_{i=0}^n K x^i \oplus \bigoplus_{i=0}^{n-1} K x^i y  \quad \text{for each $n\geq0$.}
\end{align*}
As $y^2 \in H^0(\pp_1^6) \setminus H^0(\pp_1^5)$ there are uniquely determined constants $a_0 \neq 0$, $a_1$, $a_2$, $a_3$, $a_4$ and $a_6$ such that
\[ y^2 = a_0 x^3 - a_1 xy + a_2 x^2 - a_3 y + a_4 x + a_6. \]
Transforming $x = a_0^{-1} x'$ and $y = a_0^{-1} y'$ we can normalize $a_0=1$, and so we get the Weierstrass normal form
\[ y^2 + a_1 xy + a_3 y = x^3 + a_2 x^2 + a_4 x + a_6. \]
The non-singularity of the cubic curve means that the discriminant $\De$ is non-zero. As we assume that $p=2$, we have
\[ \De = a_1^6 a_6 + a_1^5 a_3 a_4 + a_1^4 a_2 a_3^2 + a_1^4 a_4^2 + a_1^3 a_3^3 + a_3^4 \neq 0. \]
The functions $x$ and $y$ are uniquely determined up to the transformations
\[ x = \mu^2 x' + \varrho, \quad y = \mu^3 y' + \sig \mu^2 x' + \tau \quad \text{where $\mu \in K^*$ and $\varrho, \sig, \tau \in K$,} \]
which imply the transformation rules
\begin{align*}
    \mu a_1' &= a_1, \quad \mu^2 a_2' = a_2 + \sig a_1 + \varrho+ \sig^2, \\
    \mu^3 a_3' &= a_3 + \varrho a_1, \quad \mu^4 a_4' = a_4 + \sig a_3 + (\tau + \sig \varrho) a_1 + \varrho^2, \\
    \mu^6 a_6' &= a_6 + \varrho a_4 + \tau a_3 + \varrho^2 a_2 + \varrho \tau a_1 + \varrho^3 + \tau^2
\end{align*}
(see \cite[Formulae~5 and~7]{Tate74}). It follows that $\mu^{12} \De' = \De$ and so
\[ j:= \frac{a_1^{12}}{\De} \]
is an invariant of the function field $F_1|K$.

If $j=0$, i.e., $a_1=0$ then by the transformation rules we can normalize $a_2=0$, i.e.,
\[ y^2 + a_3 y = x^3 + a_4 x + a_6 \quad \text{where $\De = a_3^4 \neq 0$,} \]
and our freedom to transform is restricted by the condition $\varrho = \sig^2$.


If $j \neq 0$ then we can normalize $a_1 = 1$ and $a_3 = a_4 = 0$, i.e.,
\[ y^2 + xy = x^3 + a_2 x^2 + a_6 \quad \text{where $\De = a_6 \neq 0$ and $j = \De^{-1}$.} \]
If in addition $j$ is a square, then instead of normalizing $a_4=0$ we can normalize $a_6=0$, i.e.,
\[ y^2 + xy = x^3 + a_2 x^2 + a_4 x \quad \text{where $\De = a_4^2 \neq 0$ and $j = \De^{-1}$.} \]
In both cases it remains the freedom to transform
\[ x=x', \quad y=y'+\sigma x' \quad \text{where $\sigma \in K$}, \]
and therefore $a_2' = a_2 + \sig + \sig^2$, $a'_4=a_4$ and $a'_6=a_6$.

\begin{thm}\label{2025_01_19_23:30}
Let $F|K$ be a one-dimensional separable function field of characteristic $p=2$, of genus $g=3$ and of geometric genus $\ov g = 1$ that admits a prime $\pp$ of singularity degree $\de(\pp) = 2$, whose restriction $\pp_1$ to the first Frobenius pullback $F_1|K := F^p K |K$ is rational.
\begin{enumerate}[\upshape (a)] 
    \item \label{2025_01_19_23:31}
    If the $j$-invariant of $F_1|K$ is equal to zero, then there are functions $x$, $y$ and $z$ such that 
    \begin{align*}
        F_1 &= K(x,y) \quad \text{where $y^2 + ay = x^3 + a_4 x + a_6$, $a \in K^*$, $a_4,a_6 \in K$,} \\
        F &= F_1(z) \quad \text{where $z^2 = b x^2 + a^{-1} x + d$, $b \in K\setminus K^2$ and $d \in K$.}
    \end{align*}
    
    \item \label{2025_01_19_23:32}
    If $j \neq 0$ then there are functions $x$, $y$ and $z$ such that
    \begin{align*}
        F_1 &= K(x,y) \quad \text{where $y^2 + xy = x^3 + a x^2 + \etta x$, $a \in K$, $\etta \in K^*$, $j = \etta^{-2}$,} \\
        F &= F_1(z) \quad \text{where $z^2 = b y^2 + ab x^2 + c x$, $b \in K\setminus K^2$ and $c \in K^*$.}
    \end{align*}
\end{enumerate}
Conversely, in each item the two equations define a function field of the aforementioned type.
\end{thm}

\begin{proof}
As $\deg(\pp_1) = 1$, $[F:F_1] = p =2$ and as a singular prime is inseparable, we conclude that $\pp$ is purely inseparable of degree two and that $\pp|\pp_1$ is unramified. As $g=3$ it follows from Riemann's theorem that 
\[ \dim H^0(\pp^n) = 2n-2 \quad \text{whenever $n\geq3$.} \]
In particular $\dim H^0(\pp^3) = 4 > \dim H^0(\pp_1^3) = 3$ and so there exists a function $z \in F$ such that
\[ H^0(\pp^3) = H^0(\pp_1^3) \oplus Kz = K \oplus K x \oplus K y \oplus K z. \]
Note that $z$ is a separating variable of the function field $F|K$, in other words, the field $F$ is separable over $K(z)$, i.e., $F=F_1(z)$, i.e., $z\notin F_1$, since $H^0(\pp^3) \cap F_1 = H^0(\pp_1^3)$.
As $z^2 \in F_1 \cap H^0(\pp^6) = H^0(\pp_1^6)$ there are uniquely determined constants $b_0,b_1,b_2,b_3,c$ and $d$ such that
\[ z^2 = b_0 y^2 + b_1 xy + b_2 x^2 + b_3 y + c x + d. \]
Moreover, since $z^2$ is a separating variable of $F_1|K$, i.e., $z^2 \notin F_2 = F_1^2 {\cdot} K = K(x^2,y^2)$, it follows that
\begin{equation}\label{2025_04_08_16:30}
    (b_1,b_3,c) \neq (0,0,0).
\end{equation}
Let $t := \frac xy$. As the order of $t \in F_1$ at $\pp_1$ is equal to one, i.e., as $t$ is a local parameter at $\pp_1$, the elements of $F_1$ can be written as formal Laurent series in $t$.
As
\[ z^2 = b_0 y^2 + b_1 t y^2 + b_2 t^2 y^2 + b_3 y + c t y + d \]
the Laurent series expansion of $z^2$ can be obtained from the Laurent series expansion of $y$.
By straightforward computations we obtain
\[ y = t^{-3} + a_1 t^{-2} + a_2 t^{-1} + a_3 t^0 + (a_4 + a_1 a_3) t^1 + \cdots \]
where the dots stand for higher order terms.

\bigskip

(a) We assume that $j=0$ and normalize
\[ y^2 + ay = x^3 + a_4 x + a_6, \quad \text{$a:= a_3 \neq0$.} \]
Then $y = t^{-3} + a t^0 + a_4 t^1 + \cdots$ and so the expansion of $z^2$ starts as follows
\begin{align*}
    z^2 &= b_0 t^{-6} + b_1 t^{-5} + b_2 t^{-4} + b_3 t^{-3} + c t^{-2} + (d + a b_3 + a^2 b_0) t^0 \\
    & \qquad + (ac + a^2 b_1 + a_4 b_3) t^1 + \cdots.
\end{align*}
If $b_0 \notin K^2$ then by \cite[Proposition~4.1]{BedSt87} the assumption that $\de(\pp)=2$ means that the coefficients of the expansion of $z^2$ of order $-5$ and $-3$ are zero, while the coefficient of order $-1$ is non-zero, but this cannot happen.
Thus we conclude that $b_0 \in K^2$, and by replacing $z$ with $z + b_0^{1/2} y$ we can normalize $b_0 = 0$.
As $\pp|\pp_1$ is unramified it follows that $b_1=0$.

Next we assume, for the sake of contradiction, that $b_2\in K^2$. Then by replacing $z$ with $z + b_2^{1/2} x$ we can assume that $b_2 = 0$. As $\pp|\pp_1$ is unramified it follows that $b_3=0$, i.e., $z^2 = cx + d$, and therefore
\[ z^2 = c t^{-2} + d t^0 + act^1 + \cdots \quad \text{where $a\neq0$ and $c\neq 0$.} \]
By \cite[Proposition~4.1]{BedSt87}, if $c\notin K^2$ then $\de(\pp)=1 \neq 2$, and if $c\in K^2$ and $d\notin K^2$ then $\de(\pp) = 0 \neq 2$. If $c\in K^2$ and $d\in K^2$ then $\pp|\pp_1$ is ramified, and this is again a contradiction.

This shows that $b:= b_2 \notin K^2$. As $b_0 = b_1 = 0$ it follows from \cite[Proposition~4.1]{BedSt87} that the assumption $\de(\pp) = 2$ is satisfied if and only if the coefficients of the expansion of $z^2$ of order $-3$ and $-1$ are zero, while the coefficient of order $1$ is non-zero, that is, $b_3 = 0$ and $c\neq 0$, i.e.,
\[ z^2 = b x^2 + c x + d \quad \text{where $b \notin K^2$ and $c\neq 0$.} \]
Moreover, transforming $x = \mu^2 x'$, $y = \mu^3 y'$ and $z = z'$ where $\mu = a c$ we can normalize $c = a^{-1}$.

To prove the converse we recall that the function field $F|K = K(x,y,z) | K$ is described by the two polynomial equations
\[ y^2 + ay + x^3 + a_4 x + a_6 = 0 \, \text{ and } \, z^2 + bx^2 + cx + d = 0 \]
of degree $3$ and $2$ in $x$. By the division algorithm we get a polynomial equation of degree $1$ in $x$, from which $x$ can be eliminated
\[ x = N/D, \quad \text{$N = c (z^2 + d) + b^2 (y^2 + ay + a_6)$ and $D = b(z^2 + d) + c^2 + b^2 a_4$.} \]
It follows that 
\[ F = K(y,z) \]
and by entering the expression for $x$ into the second polynomial equation, we conclude that the generators $y$ and $z$ satisfy the sextic equation
\[ (z^2 + d) D^2 + b N^2 + c D N = 0. \]
The partial derivatives of the sextic polynomial are equal to $ab^2 c D = a b^2 c (b (z^2 + d) + c^2 + b^2 a_4)$ and $0$. Thus a point $(\bar y, \bar z)$ of the corresponding affine plane sextic curve is singular if and only if $\bar D = 0$, i.e., $b (\bar z^2 + d) = c^2 + b^2 a_4$.

To prove that $g=3$ we can assume that the constant field $K$ is separably closed. We have to show that the primes of $F|K$ different from $\pp$ are non-singular.
We assume, for the sake of contradiction, that there exists a singular prime $\fq$ different from $\pp$.
As $\fq \neq \pp$ the prime $\fq$ is centered in a point of the affine plane sextic curve.
Let $\bar x, \bar y, \bar z \in \ka(\fq)$ be the residue classes of $x$, $y$ and $z$.
As $\fq$ is a singular prime, its center is necessarily a singular point of the sextic curve, i.e., $b (\bar z^2 + d) = c^2 + b^2 a_4$, i.e., $b (b \bar x^2 + c \bar x) = c^2 + b^2 a_4$.
It follows that $\bar x$ and hence $\bar y$ belong to the separably closed constant field $K$.
Hence the prime $\fq_1$ of the elliptic function field $F_1|K$ is rational, and the function $w := x - \bar x$ is a local parameter at $\fq_1$. Note that
\[ z^2 = h w^0 + c w^1 + b w^2 \quad \text{where $h:= d + c \bar x + b \bar x^2 \in K$.} \]
By \cite[Proposition~4.1]{BedSt87}, this implies that $\de(\fq)=0$, a contradiction.

\bigskip

(b) We assume that the $j$-invariant of the first Frobenius pullback $F_1|K = K(x,\y)|K$ is non-zero, and so normalizing $a_1 = 1$ and $a_3 = a_4 = 0$ and abbreviating $a := a_2$ we obtain
\[ \y^2 + x\y = x^3 + a x^2 + \De \quad \text{where $\De \neq 0$ and $j = \De^{-1}$.} \]
Recall that $F=F_1(z)$ where, abbreviating $b:= b_0$, we have
\[ z^2 = b \y^2 + b_1 x\y + b_2 x^2 + b_3 \y + cx + d. \]
As the Laurent series expansion of $\y$ in the local parameter $t:= \frac x\y$ at $\pp_1$ starts as follows
\[ \y = t^{-3} + t^{-2} + a t^{-1} + 0 t^0 + 0 t^1 + \cdots \]
we obtain
\begin{align*}
    z^2 &= b t^{-6} + b_1 t^{-5} + (b + b_2) t^{-4} + (b_1 + b_3) t^{-3} + (a^2 b + b_2 + b_3 + c) t^{-2} \\
    &\qquad + (a^2 b_1 + a b_3 + c) t^{-1} + (a^2 b_2 + a c + d) t^0 + 0 t^1 + \cdots.  
\end{align*}
As $\pp$ is non-rational with $\de(\pp)=2$, it follows from \cite[Proposition~4.1]{BedSt87} that the coefficients of the expansion of $z^2$ of orders $-5$ and $-3$ are zero, i.e., $b_1=b_3=0$.
In turn $c\neq 0$ (see~\eqref{2025_04_08_16:30}).
Now, since the coefficient of order $-1$ is equal to $c\neq 0$ we conclude, again by using \cite[Proposition~4.1]{BedSt87}, that the assumption $\de(\pp)=2$ means that $b\notin K^2$, i.e.,
\begin{equation}\label{2025_03_01_12:30}
    z^2 = b \y^2 + b_2 x^2 + c x + d = b (xy + x^3 + a x^2 + \De) + b_2 x^2 + c x + d,
\end{equation}
where $b \notin K^2$ and $c \neq 0$.

It remains to express the fact that the genus of $F|K$ is equal to $g=3$ in terms of the coefficients $a$, $b$, $c$, $d$ and $\De$.
By Rosenlicht's genus drop formula, as $\ov g = 1$ and $\de(\pp) = 2$ this means that the primes of $F|K$ different from $\pp$ are non-singular.
Eliminating $\y$ from the preceding equation and entering into the elliptic equation $\y^2 + x\y = x^3 + a x^2 + \De$, we conclude that $F = K(x,z)$ where $x$ and $z$ satisfy the sextic equation
\[ z^4 + b x^2 z^2 + b^2 x^6 + (a^2 b^2 + b_2^2 + b b_2) x^4 + b c x^3 + (c^2 + b d) x^2 + (d + b \De)^2 = 0. \]
The primes of $F|K$ different from $\pp$ are the primes centered in the points of the affine plane sextic curve cut out by the above equation.
A singular prime is necessarily centered in a singular point of the curve.
By the Jacobian criterion, a point $(\bar x, \bar z)$ of the affine curve is singular if and only if $\bar x = 0$.
Thus a singular prime of $F|K$ different from $\pp$ is necessarily a zero of $x$. Note that there is only one prime $\fq$ of $F|K$ that is a zero of $x$.
Indeed, as the extension $F|F_1$ is purely inseparable, the uniqueness of $\fq$ follows from the observation that there is only one prime of the elliptic function field $F_1|K = K(x,\y)|K$ that is a zero of $x$, namely the prime $\fq_1$ that corresponds to the point $(0,\De^{1/2})$ of the cubic curve.
We have shown that the assumption that $g=3$ means that the prime $\fq$ is non-singular.

Next we claim that the restriction $\fq_1$ of $\fq$ to the first Frobenius pullback $F_1|K = K(x,\y)|K$ is a rational prime, i.e., $\De \in K^2$, i.e., the $j$-invariant $j = \De^{-1}$ is a square. We assume for the sake of contradiction that $\fq_1$ is non-rational, i.e., $\De\notin K^2$. However, we get the rational prime $\fq_2$ if we restrict $\fq$ to the second Frobenius pullback
\[ F_2|K = K(u,v) \quad \text{where $u:= x^2$, $v:= \y^2$ and $v^2 + uv = u^3 + a^2 u^2 + \De^2$.} \]
Note that the residue classes of $u$ and $v$ modulo $\fq_1$ satisfy $\bar u = 0$ and $\bar v = \De$.
Defining $w := v + \De$ we obtain
\[ F_2|K = K(u,w)|K \quad \text{where $w^2 + uw = u^3 + a^2 u^2 + \De u$.}\]
Moreover $w$ is a local parameter at $\fq_2$, and so the elements of $F_2$ admit formal Laurent series in $w$.
We obtain $u = \De^{-1} w^2 + \De^{-2} w^3 + \cdots$ and so by Formula~\eqref{2025_03_01_12:30}
\begin{align*}
    z^4 &= b^2 (w^2 + \De^2) + b_2^2 u^2 + c^2 u + d^2 \\
    &= (d + b \De)^2 w^0 + (b^2 + c^2 \De^{-1}) w^2 + c^2 \De^{-2} w^3 + \cdots.
\end{align*}
If the residue class $\bar z = (d + b \De)^{1/2}$ of $z$ modulo $\fq$ belongs to the residue field $\ka(\fq_1) = K(\De^{1/2})$, say $\bar z = \al + \beta \De^{1/2}$ where $\al,\beta \in K$, then by replacing $z$ with $z + \al + \be \y$ we can normalize $\bar z = 0$, hence (as $\De \notin K^2$) we deduce $\mathrm{ord}_{\fq_1}(z^2) = 1$, hence $\fq|\fq_1$ is ramified, $z$ is a local parameter at $\fq$ and so $\OO_\fq = \OO_{\fq_1}[z]$.
If the residue class $\bar z \in \ka (\fq)$ does not belong to $\ka (\fq_1)$, then we also have $\OO_\fq = \OO_{\fq_1}[z]$.
In both cases we conclude that $\de(\fq) = 2 \de(\fq_1) + 1 \geq 1$ (see \cite[Theorem~2.3]{BedSt87}), and we get the contradiction that the prime $\fq$ is singular. 

Thus we have shown that $\De$ is a square, say $\De = \etta^2$ for some $\etta \in K^*$. 
Replacing $y$ with $y + \etta$ we renormalize the variable $y$, obtaining
\begin{align*}
    F_1 &= K(x,y),  \quad \text{where $y^2 + xy = x^3 + a x^2 + \eta x$, $\eta \neq 0$, $j = \eta^{-2}$}, \\
    F &= F_1(z),  \quad \text{where $z^2 = b y^2 + b_2 x^2 + c x + d + b \eta^2$}.
\end{align*}
Now $\yy$ is a local parameter at the rational prime $\fq_1$. The power series expansion of $x$ in the parameter $\yy$ starts as follows $x = \etta^{-1} \yy^2 + \etta^{-2} \yy^3 + \cdots $. Thus
\begin{align*}
    z^2 = (d + b \etta^2) \yy^0 + (b + c \etta^{-1}) \yy^2 + c \etta^{-2} \yy^3 + \cdots.
\end{align*}
According to \cite[Proposition~4.1]{BedSt87} the condition $\de(\fq)=0$ is satisfied if and only if $(d + b \etta^2) \in K^2$, in which case we normalize $d = b \etta^2$ by subtracting $(d + b \etta^2)^{1/2}$ from $z$. 
It also follows that $\fq$ is rational if and only if $(b + c \etta^{-1}) \in K^2$.

Finally, transforming
\[ x = x', \quad \y = \y' + \sig x', \quad z' = z, \]
we obtain $a'=a+\sig + \sig^2$, $\De'=\De$, $b_2'=b_2 + \sig^2 b$, $c'=c$, and thus $b_2' + a' b' = b_2 + ab + \sig b$.
Therefore, by taking
$\sig = a + b^{-1} b_2$
we normalize $b_2 = a b$.
\end{proof}

\begin{cor}[Criterion for two function fields to be isomorphic] \label{2025_03_13_19:20}
$ $
\begin{enumerate}[\upshape (a)]
    \item \label{2025_03_13_19:21}
    If $j=0$ then the functions $x$, $y$ and $z$ are uniquely determined up to the transformations
    \[ x = \al^{-4} x' + \sig^2, \quad y = \al^{-6} y' + \sig \al^{-4} + \tau, \quad z = \al z' + \ga x' + \de \]
    where $\al \in K^*$, $\sig, \tau, \ga, \de \in K$, which imply the transformation rules
    \begin{align*}
        \al^{-6} a' &= a, \quad \al^{-8} a_4' = a_4 + \sig a + \sig^4, \quad \al^{-12} a_6' = a_6 + \sig^2 a_4 + \tau a + \sig^6 + \tau^2, \\
        \al^2 b' &= \al^{-8} b + \ga^2, \quad \al^2 d' = d + \sig^4 b + \sig^2 a^{-1} + \de^2.
    \end{align*}
    
    \item \label{2025_03_13_19:22}
    If $j\neq 0$ then the functions $x$, $y$ and $z$ are uniquely determined up to the transformations
    \[ x = x', \quad y = y' + \sig x', \quad z = \al z' + \be y' + \ga x' \]
    where $\al \in K^*$, $\sig, \be,\ga \in K$ and $\be^2 \sig^2 + (b + \be^2) \sig + \ga^2 + a \be^2 = 0$, which imply the transformation rules
    \[ \etta' = \etta, \quad a'= a + \sig + \sig^2, \quad \al^2 b' = b + \be^2, \quad \al^2 c' = c. \]
\end{enumerate}
\end{cor}

\begin{proof}
The transforms of $x$, $y$ and $z$ are linear combinations of $x', y', z', 1$ belonging to $K^* x' + K$, $K^* y' + K x' + K$ and $K^* z' + K y' + K x' + K$ that satisfy the same two equations as $x$, $y$ and $z$, while $x'$, $y'$ and $z'$ satisfy the corresponding equations with modified coefficients $a',b',c',\dots$.
\end{proof}

\begin{prop}\label{2025_02_13_10:00}
Let $F|K$ be a function field as in Theorem~\ref{2025_01_19_23:30}.
\begin{enumerate}[\upshape (a)]
    \item \label{2025_02_13_10:01}
    If $j=0$ then $F|K$ is hyperelliptic, or more precisely, $F|K$ is a quadratic extension of the function field $K(x,z)|K$ of genus zero.
    \item \label{2025_02_13_10:02}
    If $j\neq 0$ then $F|K$ is non-hyperelliptic, or more precisely, $F|K$ is the function field of the plane projective integral regular quartic curve whose projective coordinate functions are up to a proportionality factor equal to the functions $x$, $y$ and $z$, which satisfy the homogeneous quartic equation
    \[ c x^4 + (z^2 + b y^2 + a b x^2)(e(z^2 + b y^2 + a b x^2) + (y^2 + x y + a x^2)) =0, \]
    where $a,b,c,e\in K$ are constants such that $b \notin K^2$ and $(ce)^2 = j^{-1}$.
\end{enumerate}
\end{prop}

\begin{proof}
(a) The function field $K(x,z)|K$ has genus zero because its generators satisfy a quadratic equation, and in particular it is a proper subfield of $F|K$. It is a quadratic subfield, since $F=K(x,y,z)$ and $y$ satisfies a quadratic equation over $K(x,z)$. 

(b) Let $v:= \frac yx$ and $w:= \frac zx$, that is, $(1:v:w)=(x:y:z)$.
Dividing the two equations in the announcement of Theorem~\ref{2025_01_19_23:30}~\ref{2025_01_19_23:32} by $x^2$ we obtain
\begin{align*}
    v^2 + v &= x + a + \etta x^{-1}, \quad \etta\neq 0, \quad j = \etta^{-2}, \\
    w^2 &= b v^2 + a b + c x^{-1}, \quad b \notin K^2, \quad c \neq 0.
\end{align*}
We can eliminate $x$ from the second equation
\[ x = \frac c {w^2 + b v^2 + ab}. \]
As $y = vx$ and $z = wx$ it follows that 
\[ F|K = K(v,w)|K. \]
Entering the expression for $x$ into the first equation, we obtain a quartic equation between $v$ and $w$
\[c^2+\etta(w^2+bv^2+ab)^2+c(w^2+bv^2+ab)(v^2+v+a)= 0. \]
Multiplying with $x^4$ we obtain a homogenous quartic equation between $x$, $y$ and $z$.
To complete the proof we set $e:=\etta c^{-1}$.
\end{proof}

\begin{rem}\label{2025_05_10_23:55}
Let $F|K$ be a function field as in the announcement of Theorem~\ref{2025_01_19_23:30}~\ref{2025_01_19_23:32}, and let $\fq$ be the only zero of $x$.
The prime $\fq$ can be characterized by the intrinsic structure of $F|K$. It is the unique prime of $F|K$ whose restriction $\fq_1$ to the elliptic function field $F_1|K$ is the only element of order two in the abelian group of rational primes of $F_1|K$ with neutral element $\pp_1$.
Indeed, the rational primes of $F_1|K$ different from $\pp_1$ correspond bijectively to the rational points of the affine cubic curve given by the equation $y^2 + xy = x^3 + a x^2 + \etta x$, and a rational point $(\bar x, \bar y)$ of the affine cubic curve has order two if and only if its tangent line passes through the infinite point $(0:1:0)$, i.e., its tangent line is vertical, i.e., $\bar x = 0$, i.e., $(\bar x, \bar y) = (0,0)$.
The prime $\fq$ is centered at the point $(0 : 1 : b^{1/2}+e^{-1/2}) \in \PP^2(\ov K)$.

We now assume that the prime $\fq$ is rational. As mentioned at the end of the proof of Theorem~\ref{2025_01_19_23:30}~\ref{2025_01_19_23:32}, this means that $b+c\etta^{-1}$ is a square. Hence by Corollary~\ref{2025_03_13_19:20}~\ref{2025_03_13_19:22} we can normalize $c = b \eta$, i.e.,
\[ z^2 = b(y^2 + a x^2 + \etta x), \text{ i.e., } z^2 = b(x^3 + xy), \]
by taking $\al=1$, $\beta=(b+c\etta^{-1})^{1/2}$, $\sigma=a\beta^2/(b+\beta^2)$, $\ga=\beta\sigma$, and in turn our freedom to  transform is restricted by the condition $\beta=0$. The functions $x$, $y$ and $z$ are uniquely determined up to the transformations $x=x'$, $y = y' + b^{-1} \ga^2 x'$, $z = \al z' + \ga x'$ where $\al \in K^*$ and $\ga \in K$, which imply the transformation rules
\[ \etta'=\etta, \quad a'= a + b^{-1} \ga^2 + b^{-2} \ga^4, \quad \al^2 b' = b.  \]
The function field $F|K$ is isomorphic to the function field of the plane projective quartic curve defined by the equation 
\[ b^2\etta x^4 + (z^2+by^2+abx^2)(z^2+ bxy) = 0. \]
The rational prime $\fq$ is centered at the point $(0:1:0) \in \PP^2(K)$. 
An example of such a function field will be discussed in Section~\ref{2025_03_13_21:30}.
\end{rem}

\section{A class of geometrically elliptic fibrations by quartic curves}
\label{2025_10_01_15:20}

The function fields in Theorem~\ref{2025_01_19_23:30}~\ref{2025_01_19_23:32} give rise to fibrations by plane projective quartic curves of geometric genus one, via Proposition~\ref{2025_02_13_10:00}~\ref{2025_02_13_10:02}.
Let $k$ be an algebraically closed ground field of characteristic $p=2$. Consider the integral fivefold
\[ Q\su \PP^2 \times \AAA^4 \]
consisting of the pairs $((\bar x:\bar y:\bar z),(\bar a, \bar b,\bar c,\bar e)) \in \PP^2(k) \times \AAA^4(k)$ that satisfy the equation
\[ \bar c \bx^4 + (\bz^2 + \bar b \by^2 + \bar a \bar b \bx^2)( \bar e(\bz^2 + \bar b \by^2 + \bar a \bar b \bx^2) + (\by^2 + \bx \by + \bar a \bx^2)) =0, \]
which is homogeneous of degree $4$ in $\bx,\by,\bz$.
The coordinates are marked with a bar because the letters $x,y,z$ (resp., $a,b,c,e$) will stand for projective (resp. affine) coordinate functions.
The fivefold is smooth by the Jacobian criterion.
The second projection morphism 
\[ \pi:Q \tto \AAA^4 \]
is proper and its fibres are equal to
\[ \pi^{-1} (\ba,\bbb,\bc,\bar e) = Q_{(\ba,\bbb,\bc,\bar e)} \times \{ (\ba,\bbb,\bc,\bar e) \} \su Q \]
where $Q_{(\ba,\bbb,\bc,\bar e)} \su \PP^2(k)$ is the plane projective quartic curve cut out by the defining equation of the fivefold $Q$.
The fibration $\pi:Q \to \AAA^4$ is flat \cite[Theorem~III.9.9]{Har77} and defines a 4-dimensional flat family of plane projective quartic curves.
As the first projection morphism defines isomorphisms
\[ \pi^{-1} (\ba,\bbb,\bc,\bar e) \isom Q_{(\ba,\bbb,\bc,\bar e)}, \]
the fibres are often identified with the corresponding plane quartic curves.

If $\bc = 0$ then the quartic curve $Q_{(\ba,\bbb,\bc,\bar e)}$ consists of an integral component and a non-reduced one, which is a  double line. The two components intersect in two different points.
Now we restrict the base of the fibration to the affine open subset
\[ \AAA^4_c = \{ (\ba,\bbb, \bc,\bar e) \mid \bc\neq0 \} = \AAA^4 \setminus \{\bc=0\}\su \AAA^4. \]
Though the total space $Q_c=\pi^{-1}(\AAA^4_c)$ of the fibration $Q_c\to \AAA^4_c$ is smooth, each fibre $\pi^{-1}(\ba,\bbb,\bc,\bar e)$ with $\bc\neq0$ has over the point $(0:1:\bbb^{1/2})$, or more precisely, at the point $((0:1:\bbb^{1/2}),(\ba,\bbb,\bc,\bar e))$, a singularity of multiplicity two.
By blowing up two times over the singular point we can see that the curve singularity is unibranch of singularity degree $\de=2$.
If $\bc \bar e \neq0$ then the unibranch singular point is the unique singular point of the fibre, and so by Bezout's theorem for plane projective curves the curve is integral, and by the genus-degree formula for plane curves it has arithmetic genus $g =(4-1)(4-2)/2=3$ and geometric genus $\ov g = g -\de = 1$.
Thus if we restrict the base of the fibration to the affine open subset
\[ \AAA^4_{ce} =\{ (\ba,\bbb,\bc,\bar e) \mid \bc\bar e \neq0\} = \AAA^4_c \setminus \{\bar e=0\} \su \AAA^4 \]
then we get a fibration $Q_{ce} \to \AAA^4_{ce}$, whose fibres are integral plane projective quartic curves of geometric genus one.
However, if $\bc\neq0$ but $\bar e=0$, then the fibres are also integral, but they admit over the point $(0:0:1) \in \PP^2$ a node as a second singularity, and so the geometric genus of the fibres is zero, i.e., the fibres are rational curves.

Each fibre of the fibration $Q_c\to \AAA^4_c$ is \emph{strange}; indeed the tangents at the smooth points have a common intersection point, namely $(0:0:1)$.
However, the tangent of the unibranch singularity does not pass through this point; 
it is given by the equation 
$\bz=\bbb^{1/2}\by+\ba^{1/2}\bbb^{1/2}\bx$ 
and intersects the quartic only at the singular point $(0:1:\bbb^{1/2})$.
If $\bar e=0$ then the common intersection point is the node of the fibre, and so, as the fibre is a plane quartic curve, there are no inflection points on this fibre.

Each tangent at a smooth point of a fibre of $Q_{ce} \to \AAA^4_{ce}$ is a bitangent or it passes through the singular point $(0:1:\bbb^{1/2})$, 
and so the fibre, as it is a plane quartic curve, does not admit ordinary inflection points.
However the two tangency points may coincide, that is, the bitangent may become a non-ordinary inflection tangent.
This happens at a point $(\bx:\by:\bz)$ of the quartic if and only if $\by^2 + \bx \by + \ba \bx^2 = 0$, which implies by the defining equations of $Q_{(\ba,\bbb,\bc,\bar e)}$ that $\bz = (\ba^2 \bbb^2   + \bar c \bar e^{-1} )^{1/4} \bx + \bbb^{1/2} \by$.
Thus each fibre of $Q_{ce} \to \AAA^4_{ce}$ has two non-ordinary inflection points.
If $\bar e$ tends to zero then the two inflection points collide at the nodal singularity of the fibre $Q_{(\ba,\bbb,\bc,0)}$.
If $\bc$ tends to zero then the two inflection points approach the two intersection points of the components of the reducible fibre $Q_{(\ba,\bbb,0,\bar e)}$.

For each pair of constants $(\ba,\bar\eta) \in k^2$ we consider the integral plane projective cubic curve
\[ E_{(\ba,\bar \eta)} := \{ (\bx:\by:\bz) \in \PP^2(k) \mid \by^2 \bz + \bx \by \bz = \bx^3 + \ba \bx^2 \bz + \bar \eta \bar x \bz^2 \} \] 
which we mark at the inflection point $(0:1:0)$.
If $\bar \eta \neq0$ then the cubic curve is an elliptic curve with $j$-invariant $\bar \eta^{-2}$.
If $\bar \eta=0$ then it admits a nodal singularity at the point $(0:0:1)$, and so it is a rational curve.
We consider the integral fivefold
\[ E \su \PP^2 \times \AAA^4 \]
consisting of the pairs $((\bx : \by : \bz),(\bar a, \bar b, \bar c, \bar e)) \in \PP^2 \times \AAA^4$ that satisfy the defining equations of $E_{(\bar a, \bar \eta)}$, where we set $\bar\eta = \bar c \bar e$.
The second projection morphism
\[ \eps: E \tto \AAA^4 \]
is proper and flat, its fibres are equal to
\[ \eps^{-1} (\bar a, \bar b, \bar c, \bar e) = E_{(\bar a, \bar c \bar e)} \times \{ (\bar a, \bar b, \bar c, \bar e)\}, \]
hence isomorphic to $E_{(\bar a, \bar c \bar e)}$, and so it is an elliptic fibration.

The function field $F|K=k(Q)|k(\AAA^4)$ of the fibration $Q \overset \pi \to \AAA^4$ is the function field described in Theorem~\ref{2025_01_19_23:30}~\ref{2025_01_19_23:32} and Proposition~\ref{2025_02_13_10:00}~\ref{2025_02_13_10:02}, where $a,b,c,e$ are the coordinate functions  of $\AAA^4$.
Its first Frobenius pullback $F_1|K$ is the function field $k(E)|k(\AAA^4)$ of the elliptic fibration $E \overset \eps\to \AAA^4$. 

\begin{prop}\label{2025_10_26_15:00}
The purely inseparable degree two field extension $k(Q)|k(E)$ defines a morphism
 $Q \to E$ over $\AAA^4$ which induces on the fibres the morphisms
\[ Q_{(\bar a, \bar b, \bar c, \bar e)} \tto E_{(\bar a, \bar c \bar e)} \] 
defined by the assignments 
\[ (\bar x : \bar y : \bar z) \mapsto
(\bar c \bar x^2 : \bar c \bar x \bar y : \bar z^2 + \bar b \bar y^2 + \bar a \bar b \bar x^2) \ \text{ or } \
 (\bar x h : \bar y h : \bar x^3) \]
where $h:= 
\bar e (\bar z^2 + \bar b \bar y^2 + \bar a \bar b \bar x^2) 
+ \bar y^2 + \bar x \bar y + \bar a \bar x^2 .$
The image of the singular point $(0:1:\bar b^{1/2})$ of the quartic is equal to the marked point $(0:1:0)$ of the cubic.
If $\bar c \neq0$ then the morphism $Q_{(\bar a, \bar b, \bar c, \bar e)} \to E_{(\bar a, \bar c \bar e)}$ is a bijection.
\end{prop} 

\begin{proof}
If we define $(1:v:w) := (x:y:z)$ then as in the proof of Proposition~\ref{2025_02_13_10:00} we get 
\[ (x:y:1) = (c:cv:w^2+bv^2+ab). \]
Thus the field extension $k(Q)|k(E)$, which is equal to $k(v,w)|k(x,y)$, defines the rational map $Q \ttto E$ given by the first assignment of the announcement.
Clearly, this rational map is defined at the points on $Q$ with $\bar z^2 + \bar b \bar y^2 + \bar a \bar b \bar x^2 \neq 0$ or $\bc\bx \neq 0$.
Due to the defining equation of the fivefold $Q$, the rational map is also given by the second assignment,
and so it is also defined at the points with $\bar x \neq 0$ or $\by h \neq 0$. 
If $\bar z^2 + \bar b \bar y^2 + \bar a \bar b \bar x^2 = 0$ and $\bar x = 0$ then the point on the fibre becomes $(0 : 1 : \bar b^{1/2})$, and by the second assignment its image is defined and equal to  $(0:1:0)$. 
Hence the rational map is a morphism $Q \to E$.

If $\bc\neq 0$ then it is easy to check that the morphism $Q_{(\bar a, \bar b, \bar c, \bar e)} \to E_{(\bar a, \bar c \bar e)}$ is bijective and that its inverse is given by the assignment $\qedsymbol$
\[
(\bx:\by:\bz) \mapsto \begin{cases}
(\bx:\by:\bbb^{1/2}\by+\ba^{1/2}\bbb^{1/2}\bx
+\bc^{1/2}\bx^{1/2}\bz^{1/2})  &\text{if } (\bx:\by:\bz) \neq (0:0:1), \\
(0:\bar e^{1/2}:1+\bbb^{1/2}\bar e^{1/2}) &\text{if } (\bx:\by:\bz)=(0:0:1).
\end{cases} \qedhere
\]  
\end{proof} 
If $\bar c=0$ then the integral component $\{h=0\}$ of the quartic $Q_{(\ba,\bbb,0,\bar e)}$ is mapped to the node $(0:0:1)$ of the cubic $E_{(\ba,0)}$, while the restriction of the morphism to the line 
$\{ \bar z=\bbb^{1/2} \by+ \ba^{1/2} \bbb^{1/2} \bx \}$,
which is the support of the non-reduced component,
is the desingularization morphism of the nodal cubic $E_{(\ba,0)}$.

\begin{prop}\label{2025_03_01_20:00}
Let $\bar \eta = \bc \bar e$ and $\bar c \neq 0$. 
Then there is a bijective morphism
\[ E_{(\ba^{1/2},\bar \eta^{1/2})} \tto Q_{(\ba,\bbb,\bc,\bar e)} \]
defined by 
\[ (\bx':\by':\bz')\mapsto 
\begin{cases} 
(\bx'^2:\by'^2:\bbb^{1/2} \by'^2 + \ba^{1/2} \bbb^{1/2} \bx'^2 + \bc^{1/2} \bx' \bz') & \text{if } (\bx':\by':\bz') \neq (0:0:1), \\
 (0:\bar e^{1/2}:1+\bbb^{1/2}\bar e^{1/2}) & \text{if } (\bx':\by':\bz')=(0:0:1).
\end{cases}\]
It is birational, or more precisely, it defines an isomorphism
\[ E_{(\bar a^{1/2},\bar\eta^{1/2})} \setminus \{ (0:1:0) \} \overset\sim\tto Q_{(\bar a, \bar b, \bar c, \bar e)} \setminus \{ (0:1:\bbb^{1/2}) \}. \] 
The composite morphism 
$E_{(\ba^{1/2},\bar\eta^{1/2})}\to Q_{(\ba,\bbb,\bc,\bar e)} \to E_{(\ba,\bar\eta)}$ 
is the Frobenius morphism that raises coordinates to their $p$-th powers (where $p=2$).
\end{prop}
\begin{proof} 
The inverse of the bijective morphism 
$ Q_{(\ba,\bbb,\bc,\bar e)} \to E_{(\ba,\bar \eta)} $ 
of Proposition~\ref{2025_10_26_15:00}
is not a morphism and not even a rational map. However, composing it with the Frobenius morphism 
$E_{(\bar a^{1/2},\bar\eta^{1/2})} 
{\to} E_{(\ba,\bar\eta)}$
we get the map $E_{(\ba^{1/2},\bar\eta^{1/2})} \to Q_{(\ba,\bbb,\bc,\bar e)}$ 
of the announcement. This is a rational map. To prove that it is a morphism, it remains to check that it is defined at the point $(0:0:1)$.
Taking $\bar z' = 1$ and using that $\bar y'^2 = \bar x' \bar y' + \bar x'^3 + \bar a^{1/2} \bar x'^2 + \bar \eta^{1/2} \bar x'$ we can divide the homogeneous coordinates of the image points by $\bar x'$, to conclude that the rational map is also given by the assignment
\[ (\bar x' : \bar y' : 1) \mapsto 
(\bar x' : \bar y' + \bar x'^2 {+} \bar a^{1/2} \bar x' + \bar \eta^{1/2} : \bar b^{1/2} \bar y' + \bar b^{1/2} \bar x'^2 + \bar b^{1/2} \bar \eta^{1/2} + \bar c^{1/2} ). \] 
It now follows that the map 
$E_{(\ba^{1/2},\bar\eta^{1/2})} \to Q_{(\ba,\bbb,\bc,\bar e)}$ 
is a morphism and that the image of the point $(0:0:1)$ is equal to 
$(0 : \bar e^{1/2} : 1+\bar b^{1/2}\bar e^{1/2}).$

By construction, the composite morphism
$  E_{(\bar a^{1/2},\bar \eta^{1/2})} \to Q_{(\bar a, \bar b, \bar c, \bar e)} \to E_{(\bar a, \bar \eta)} $ 
is the Frobenius morphism defined by $(\bar x' : \bar y' : \bar z') \mapsto (\bar x'^2 : \bar y'^2 : \bar z'^2)$.
Thus the inverse map
\[ Q_{(\bar a, \bar b, \bar c, \bar e)} \ttto E_{(\bar a^{1/2},\bar \eta^{1/2})} \]
is given by $(\bar x : \bar y : \bar z) \mapsto (\bar c^{1/2} \bar x : \bar c^{1/2} (\bar x \bar y)^{1/2} : \bar z + \bar b^{1/2} \bar y + \bar a^{1/2} \bar b^{1/2} \bar x )$.
In order to check that it is indeed a rational map and that it is regular at each point except $(0 : 1 : b^{1/2})$, we eliminate the product $\bar x \bar y$ from the defining equation of $Q$ and write $(\bar x \bar y)^{1/2}$ as a rational expression in $\bar x$ and $\bar y$:
\[ (\bar x \bar y)^{1/2} = \frac{\bar c^{1/2} \bar x^2}{\bar z + \bar b^{1/2} \bar y + \bar a^{1/2} \bar b^{1/2} \bar x} + \bar y + \bar a^{1/2} \bar x + \bar e^{1/2} (\bar z + \bar b^{1/2} \bar y + \bar a^{1/2} \bar b^{1/2} \bar x). \qedhere \]
\end{proof}

\begin{cor}\label{2025_05_11_00:45}
Assume that $\bar \eta = \bar c \bar e \neq 0$, i.e., the point
$(0 : 1 : \bar b^{1/2})$ is the unique singular point of
the quartic curve $Q_{(\bar a, \bar b, \bar c, \bar e)}$. 
Then $E_{(\bar a^{1/2},\bar \eta^{1/2})} \to Q_{(\bar a, \bar b, \bar c, \bar e)}$ is the desingularization morphism of $Q_{(\bar a, \bar b, \bar c, \bar e)}$.
\end{cor}

To make a more detailed study of the singularity on the quartic $Q_{(\ba,\bbb,\bc,\bar e)}$ let $x'$, $y'$ and $z'$ be the projective coordinate functions of the cubic $E_{(\bar a^{1/2}, \bar \eta^{1/2})}$.
Then $t' := \frac{x'}{y'}$ is a local parameter of the point $(0:1:0)$ and the formal expansion of $s' := \frac{z'}{y'}$ starts as follows
\[ s' = t'^3 + t'^4 + (1+\ba^{1/2}) t'^5 + t'^6 + (1 + \ba^{1/2} + \ba + \bar \eta^{1/2}) t'^7 + \cdots \]
(see \cite[Formula 13]{Tate74}).
As $(x' : y' : z') = (t' : 1 : s')$ the morphism $E_{(\bar a^{1/2}, \bar \eta^{1/2})} \to Q_{(\bar a, \bar b, \bar c, \bar e)}$ is also given by the assignment
\[ (\bar t' : 1 : \bar s') \mapsto (\bar t'^2 : 1 : \bbb^{1/2} + \ba^{1/2} \bbb^{1/2} \bar t'^2 + \bc^{1/2} \bar t' \bar s' ), \]
and so the formal expansion of the singularity at $(0 : 1 : \bar b^{1/2})$ starts as follows
\[  (\bar t'^2 : 1 : \bbb^{1/2} + \ba^{1/2} \bbb^{1/2} \bar t'^2 + \bc^{1/2} \bar t'^4 + \bc^{1/2} \bar t'^5 + \cdots ). \]
In this way one can see again that the singularity is unibranch of multiplicity two and of singularity degree two.

If $\bar\eta = \bar c \bar e \neq 0$ then the cubic $E_{(\bar a,\bar\eta)}$ is an elliptic curve, and so it is an abelian group whose neutral element is the marked point $(0:1:0)$. 
Via the bijection 
$Q_{(\bar a, \bar b, \bar c, \bar e)} \leftrightarrow E_{(\ba,\bar \eta)}$ 
or
$E_{(\ba^{1/2},\bar \eta^{1/2})}\leftrightarrow Q_{(\bar a, \bar b, \bar c, \bar e)}$ 
of Proposition~\ref{2025_10_26_15:00} and \ref{2025_03_01_20:00}, 
the quartic $Q_{(\bar a,\bar b,\bar c,\bar e)}$ becomes an abelian group whose neutral element is its unique singular point $(0:1:\bar b^{1/2})$.

\begin{prop}\label{2025_10_13_16:15}
Assume that $\bar\eta = \bar c \bar e \neq 0$.
\begin{enumerate}[\upshape(i)]
\item
The abelian group $Q_{(\bar a,\bar b,\bar c,\bar e)}$ has a unique element of order two, which is its unique smooth point 
$(0:1:\bar b^{1/2}+\bar e^{-1/2})$ on the line $\{\bx=0\}$ spanned by the singular point $(0:1:\bar b^{1/2})$ and the strange point $(0:0:1)$. 

If $\bar e$ tends to zero then the element of order two approaches the node of the quartic $Q_{(\ba,\bbb,\bc,0)}$. 
If $\bc$ tends to zero then the element of order two becomes the unique intersection point of the line $\{\bx=0\}$ and the integral component of $Q_{(\ba,\bbb,0,\bar e)}$.

\item
The elements of order four of the quartic $Q_{(\bar a, \bar b, \bar c, \bar e)}$ are precisely the two (non-ordinary) inflection points. 
\end{enumerate}
\end{prop} 


\begin{proof} 
(i) The point $(0:0:1)$ is the unique point of the elliptic curve $E_{(\ba,\bar \eta)}$ that is different from the neutral element $(0:1:0)$ and whose tangent line passes through the neutral element, 
and so it is the unique element of $E_{(\ba,\bar \eta)}$ of order two. The corresponding point of the quartic is equal to 
$(0:1:\bar b^{1/2}+\bar e^{-1/2})$ 
(see also Remark~\ref{2025_05_10_23:55}). 

(ii) A point $(\bar x : \bar y : \bar z)$ on the cubic $E_{(\ba,\bar \eta)}$ has order four if and only if it is different from the point $(0:0:1)$ of order two, and its tangent line passes through this point, i.e., $\bx\neq 0$ and $\bar y^2 + \bar x \bar y + \bar a \bar x^2 = 0$. 
By Proposition~\ref{2025_10_26_15:00}, as $\bx\neq 0$ the first two homogeneous coordinates of the corresponding points of the quartic 
$Q_{(\bar a,\bar b,\bar c,\bar e)}$
are up to proportionality equal to $\bx$ and $\by$. 
Hence a point $(\bar x : \bar y : \bar z)$ of the quartic 
 has order four if and only if 
$\bar y^2 + \bar x \bar y + \bar a \bar x^2=0$. 
The condition $\bx\neq 0$ becomes redundant, because $\bx=0$ would imply $\by=\bz
=0$.
Hence a point of the quartic has order four if and only if it is a (non-ordinary) inflection point.
\end{proof}

Let $F|K=K(x,y,z)|K$ be a function field as in Proposition~\ref{2025_02_13_10:00}~\ref{2025_02_13_10:02}, that is, 
\[c x^4 + (z^2 + b y^2 + ab x^2) (e (z^2 + b y^2 + a b x^2) +y^2+xy+ax^2) = 0,\]  
$\mathrm{char}(K)=2, \
a \in K, \,\ b \in K\setminus K^2 \text{ and } c,e \in K^*.$  
Assume that $K$ is the field $k(B)$ of rational functions on an algebraic variety $B$ defined over the ground field $k$. 
We want to realize $F|K$ as the function field of some fibration over the variety $B$, i.e., 
we look for an algebraic variety $T$ and a proper surjective morphism $\Phi: T \to B$ whose function field $k(T)|k(B)$ is equal to $F|K$. The function field  determines the fibration $\Phi: T \to B$ uniquely up to birational equivalence, 
where a second fibration $\breve \Phi : \breve T \to \breve B$ is called birationally equivalent to $\Phi$ if and only if there is a birational map $\breve  T \ttto T$ between the total spaces together with a birational map $\breve B \ttto B$ between the bases such that the square
\[
\begin{tikzcd}
    \breve T \ar[r,dashed] \ar[d,"\breve \Phi"] & T \ar[d,"\Phi"] \\
    \breve B \ar[r,dashed]  & B
\end{tikzcd}
\]
commutes. 

By restricting if necessary the variety $B$ to a dense open subset, we assume that the rational functions $a,b,c,e \in k(B)$ are regular functions on $B$.
Thus there is given an algebraic variety $B$ equipped with four regular functions $a,b,c,e \in k[B]$ that satisfy $c\neq 0, e\neq 0$ and $b\notin k(B)^2$. 
Now we are able to define a fibration 
\[ \Phi_B: T_B \tto B \]
with function field $k(T_B)|k(B) = F|K$ as follows.
The total space is the hypersurface
$ T_B \su \PP^2(k) \times B  $ 
cut out by the defining equation of the function field $F|K$, i.e., 
\[ T_B :=\{((\bar x:\bar y:\bar z),P)\in \PP^2(k)\times B \mid c(P) \bar x^4 + \dots =0 \}.\]
The morphism $\Phi_B: T_B \tto B$ is defined to be the second projection morphism.
This is a fibration by plane projective quartic curves with function field $F|K$
 whose fibers 
 \[ \Phi_B^{-1}(P) \cong Q_{(a(P),b(P),c(P),e(P))}\]
 have been studied in the first part of this section. 
The moving singularity of the fibration is the prime divisor of $T_B$ consisting of the points $((0:1:b(P)^{1/2}),P)$ where $P\in B$. 

As $b$ is assumed to be a regular function on $B$ and as $b\notin k(B)^2$ and hence $b$ is non-constant, the base $B$ cannot be a proper variety.
However there are possibilities to enlarge the base.
\begin{rem}
 If we cancel our assumption that the rational function $a,b,c,e$ are regular on $B$, then the fibre of a point $P \in B$ can be still defined if the fractional ideal of the local ring $\OO_P$ generated by the coefficients of the quartic polynomial is principal, i.e.,
\[ \OO_P \langle 1,a,b,ab, c + a^2 b + a^2 b^2, e, b^2 e  \rangle = \OO_P h_P \quad \text{for some $h_P \in K^*$.} \]
In this case the fibre of the point $P$ is given by the homogeneous quartic polynomial expression $ (c x^4+ \dots)/h_P \in \OO_P[x,y,z]$.
In particular, if the base $B$ is a smooth curve (as in Section~\ref{2025_03_13_21:30}) then, as its local rings are principal ideal domains, the fibre of each point is defined. 
 \end{rem} 
Our assumption that the functions $a,b,c$ and $e$ are regular means that $h_P=1$ for each $P\in B$. 
If $B$ is a closed subset of $\AAA^4$ equipped with its four coordinate functions, then $\Phi_B$ is a closed subfibration of the fibration $\pi: Q\to \AAA^4$ defined at the beginning of this section. In Section~\ref{2025_03_13_21:30} we will study the pencil of quartics whose base $B$ is the hyperbola 
in $\AAA^4$ cut out by the equations $a=0$, $c=1$ and $be=1$.

\begin{prop} 
The fibration $\Phi_B:T_B\to B$ is a dominant base extension of a closed subfibration of $\pi: Q \to \AAA^4$.
\end{prop}

\begin{proof} 
Let $\check B$ be the closed subset of $\AAA^4$ corresponding to the integral $k$-algebra $k[a,b,c,e]$, or more precisely, let $\check B\subseteq\AAA^4$ be the locus of the polynomials in four variables that vanish on $(a,b,c,e)$, and let $\check a,\check b, \check c, \check e$ be its coordinate functions.
As the regular functions $a,b,c,e$ define a dominant morphism $B\to \check B$, the fibration $\Phi_B$ is a dominant base extension of the closed subfibration $\Phi_{\check B}$  
of the fibration $\pi: Q\to\AAA^4$.
\end{proof}

Note that the condition $ce\neq 0$ is invariant under dominant base extensions, and so it is equivalent to $\check{c} \check{e} \neq 0.$ 
Moreover, the condition $b \notin k(B)^2$ implies $\check b \notin k(\check B)^2$.
The converse is true if the dominant morphism $B\to \check B$ is separable. 

By restricting the base of the fibration $\Phi_B: T_B\to B$ to a dense open subset of $B$, we will arrange that the total space becomes smooth and that the special fibres behave like the geometric generic fibre. 
To this end, 
we note that the condition $b \notin k(B)^2$ means that the differential $d_{B|k}(b)$ is non-zero.

\begin{prop}
If the base of the fibration $\Phi_B\to B$ is restricted to the dense open subset 
$B_{ce}=\{P\in B\mid c(P)\neq 0,\, e(P)\neq 0\}$, 
then the fibre $\Phi_B^{-1}(P)$ over each point $P\in B_{ce}$ has a unique singularity at $((0:1:b(P)^{1/2}),P)$, which is unibranch of multiplicity two and singularity degree two. 
If the base is further restricted to the dense open subset 
\[\breve B_{ce}:= \{P\in B_{smooth}\mid c(P)\neq 0,\, e(P)\neq 0,\, d_{B|k}(P)\neq 0\}\] 
then the total space $\Phi_B^{-1}(\breve B_{ce})$ becomes smooth.
\end{prop}

\begin{proof} If $c(P)\neq 0$, then the point $((0:1:b(P)^{1/2}),P) \in \PP^2 \times B$ is a non-smooth point of the fibre $\Phi_B^{-1}(P)$, which is unibranch of multiplicity two and singularity degree two. If in addition $e(P)\neq 0$ then the remaining points of the fibre are smooth points of the fibre and in particular they are also smooth points of the total space.

Let $f(x,y,z,a,b,c,e) = c x^4 + \dots$ be the polynomial expression that defines the fibration $\Phi_B$.
As the partial derivatives $\frac{\pa f}{\pa a}$, $\frac{\pa f}{\pa b}$, $\frac{\pa f}{\pa c}$, $\frac{\pa f}{\pa e}$ at the point $((0,1,b(P)^{1/2}),P) \in \AAA^3\times B$ are equal to $0,1,0,0$, we conclude that the unique non-smooth point $((0:1:b(P)^{1/2}),P)$ of the fibre $\Phi_B^{-1}(P)$ is a smooth point of the total space if and only if $d_{B|k}(b)(P) \neq 0$.
\end{proof}

We have obtained the geometric meaning of the three conditions $b\notin K^2$, $c\neq0$ and $e\neq 0$ 
in Theorem~\ref{2025_01_19_23:30}~\ref{2025_01_19_23:32}: there exists a dense open subset of the base $B$, over which the total space $T_B$ is smooth and the fibres are integral curves of geometric genus one. 

We will compare the fibration by quartic curves $\Phi_B:T_B \to B$ with the elliptic fibration $\eps_{B}: E_B\to B$ where
\[ E_B :=\{((\bx:\by:\bz),P)\in\PP^2\times B\mid\by^2\bz+\bx\by\bz=\bx^3+a(P)\bx^2\bz+c(P)e(P)\bx\bz^2\}\]
and where $\eps_B$ is the second projection morphism. 
It is a dominant base extension of a closed subfibration of the elliptic fibration $\eps:E\to\AAA^4.$ 
The fibres of the open subfibration $\eps_{B_{ce}}:E_{B_{ce}} \to B_{ce}$ are the elliptic curves $E_{(\ba,\bc\bar e)}$ with the $j$-invariants $c(P)^{-2}e(P)^{-2}$, and the total space is smooth. 
By Proposition~\ref{2025_10_26_15:00}, the bijective purely inseparable degree-two morphism $T_{B_{c}}\to E_{B_{c}}$ 
over the base $B_{c}$ is given on the fibres by the assignment
\[ (\bx:\by:\bz) \mapsto \begin{cases}
(c(P)\bx^2:c(P)\bx\by:\bz^2{+}b(P)\by^2{+}a(P)b(P)\bx^2) 
 &\text{if } (\bx:\by:\bz) \neq (0:1:b(P)^{1/2}), \\
(0:1:0) &\text{if } (\bx:\by:\bz) = (0:1:b(P)^{1/2}).
\end{cases}
\]

We will represent the fibration $\Phi_{B_c}:T_{B_c}\to B_c$ as a quotient of an elliptic fibration, obtained by simultaneously resolving the singularities of the fibres. 
To get rid of the rational exponents in the description of the desingularisation morphisms (see Proposition~\ref{2025_03_01_20:00}) we will introduce the \emph{Frobenius pushforward} of an algebraic variety.

Viewing the algebraic variety $B$ as a $\Spec(k)$-scheme, we note that its relative Frobenius endomorphism $\varphi_k:B\to B$, which raises coordinates to their $p$-th powers, is only defined if $B$ is defined over the prime field $\F_p$.
In contrast, the absolute Frobenius endomorphism $\vf:B\to B$ is always defined, but in general it is not a $\Spec(k)$-morphism. 
However, replacing the target space $B$ by its Frobenius pullback $B^{(p)}$ one obtains a $\Spec(k)$-morphism $B \to B^{(p)}$. 
Instead, we replace the source.
We define
the \emph{Frobenius pushforward} $B'$ of $B$ by modifying the structure morphism $B\to \Spec(k)$ to $B \to \Spec(k) \overset \vf\to \Spec(k)$, and then the endomorphism $B \overset \vf\to B$ gives rise to a $\Spec(k)$-morphism $B' \to B$ we also denote by $\vf$.
Note that the local sections of the structure sheaf of the Frobenius pushforward $B'$ (resp., of the Frobenius pullback $B^{(p)}$) are the $p$-th roots (resp., the $p$-th powers) of the local sections of the structure sheaf of $B$. 
In this sense, the Frobenius pushforward $B'$ can be viewed as the $(-1)$-th Frobenius pullback of the $\Spec(k)$-scheme $B$.

If $B$ is an affine variety, say a closed subvariety of $\AAA^n$, then its Frobenius pushforward can be realized by the affine variety
\[ B' := \{ (x_1,\dots,x_n) \in \AAA^n(k) \mid h(x_1^p,\dots,x_n^p) = 0 \text{ for each $h\in I(B)$} \} \]
where $I(B)$ denotes the prime ideal of the polynomials that vanish on $B$, and where the $k$-morphism $B'\overset{\varphi}\to B$ is given by the Frobenius map $(x_1,\dots,x_n) \mapsto (x_1^p,\dots,x_n^p)$.
More generally, a similar observation applies if $B$ is a quasi-projective variety, say a locally closed subset of $\PP^n$.

In the setting of this section, where the variety $B$ is equipped with the four regular functions $a, b, c \text{ and } e$, we equip the Frobenius pushforward $B'$ with the four regular functions $a'=a\circ\varphi,\ b'=b\circ\varphi,\ c'=c\circ\varphi\ \text{ and }\ e'=e\circ\varphi$ where $\varphi:B'\to B$ is the Frobenius pushforward. 
Then the fibration $\eps_{B'}:E_{B'}\to B'$ is defined and its total space is equal to 
\[E_{B'}=\{((\bx':\by':\bz'),P')\in \PP^2 \times B'\mid
\by'^2\bz'+\bx'\by'\bz'=\bx^3+a'(P')\bx'^2\bz'+c'(P')e'(P')\bx'\bz'^2\}.\] 
By restricting the bases 
and summarizing results of this section we obtain: 
\begin{thm}
The fibration $\Phi_{B_c}:T_{B_c}\to B_c$ by quartic curves admits a cover by the elliptic fibration
$
\eps_{B'_{c'}}: 
E_{B'_{c'}} \tto B'_{c'}\,.
$
This cover is purely inseparable of exponent one on the bases but birational on the fibres. More precisely, there is a commutative diagram
\[
\begin{tikzcd}
  E_{B'_{c'}} \ar[r] \ar[d,"\eps"] & T_{B_c} \ar[d,"\Phi"] \ar[r] & E_{B_c} \ar[d,"\eps"] \\
   B'_{c'} \ar[r,"\varphi"] & B_c \ar[r,"id"] & B_c
\end{tikzcd}
\]
where the morphism on the left-hand side is given
on the bases by the bijective purely inseparable Frobenius morphism $\varphi$ and
on the fibers by the bijective birational desingularization morphisms of Proposition~\ref{2025_03_01_20:00}. 
It assigns to a point $P'$ of the base $B'_{c'}$ the point $P:=\vf(P')\in B_c\,$, and assigns to each point $((\bx':\by':\bz'),P')$ of the fibre over $P'$ the point 
\[\begin{cases}
 ((\bx'^2:\by'^2:b(P)\by'^2+a(P)b(P)\bx'^2+c(P)\bx'\bar z'),P) & \text{if $(\bx':\by':\bz')\neq (0:0:1)$}, 
 \\((0:e(P):1+b(P)e(P)),P) & \text{if $(\bx':\by':\bz')=(0:0:1)$}.
\end{cases}\] 
In contrast, the morphism on the right-hand side leaves the points of the base fixed and is given on the fibres by the bijective purely inseparable degree-two morphism of Proposition~\ref{2025_10_26_15:00} as described before.
The composite horizontal morphisms are Frobenius morphisms that raise coordinates to their second powers.
\end{thm}


\section{A pencil of geometrically elliptic quartic curves \\ in characteristic two}\label{2025_03_13_21:30}

In this section we take an in-depth look at the pencil of quartics obtained from the closed subfibration of
$\pi: Q \to \AAA^4$ over the hyperbola $\{ a=0, \, c=1, \, be = 1 \} \su \AAA^4_{(a,b,c,e)}$.
Let $k$ be an algebraically closed ground field of characteristic two.
Let $S$ be the projective surface in $\PP^2 \times \PP^1$ defined by the bihomogeneous equation
\[ s t x^4 + (s z^2 + t y^2) (s z^2 + t x y) = 0  \]
of degrees 4 and 2 in the projective coordinates $x,y,z$ and $s,t$ of $\PP^2$ and $\PP^1$.
By the Jacobian criterion, the surface $S$ has precisely two singular points
\[ M:= ((0:1:0),(1:0)), \quad N:= ((0:0:1),(0:1)), \]
which are normal singularities by Serre's conditions \cite[23.8]{Mat89}.
The projection morphism
\[ \phi:S\tto \PP^1\]
yields a flat fibration by plane projective quartic curves.
The fibres over the points that are different from $(1:0)$ and $(0:1)$ have been described in Section~\ref{2025_10_01_15:20};
indeed, in the notation of that section the fibre over a point $(1:b)$ with $b\neq0$ is the quartic curve $Q_{(0,b,1,b^{-1})}$.
The fibre over the point $(1:0)$ is a quadruple line whose support is equal to
\[ Z:= \{z=0\} \times \{ (1:0) \} \su S. \]
Over the point $(0:1)$ the fibre is a reducible non-reduced curve consisting of two lines of multiplicities one and three, with supports
\[ X:= \{x=0\} \times \{ (0:1) \}, \quad Y:= \{y=0\} \times \{ (0:1) \}. \]

By intersecting the surface $S$ with the hyperplane $\{x=0\} \su \PP^2 \times \PP^1$ we obtain two horizontal prime divisors on $S$, namely the smooth rational curves
\[ \pp:=\{x=0, \, s z^2 + t y^2 = 0\}, \quad \fq:= \{(0:1:0)\} \times \PP^1. \]
The horizontal prime divisor $\pp$ is purely inseparable of degree 2 over the base $\PP^1$,
and it cuts each fibre $S_{(1:b)}$ (where $b\neq0$) at its only singular point $((0 : 1 : b^{1/2}),(1:b))$.
The horizontal divisor $\fq$ has degree 1 over the base $\PP^1$, and, as a prime of the function field $F|K=k(S)|k(\PP^1)$ of the fibration, it is the rational prime mentioned in Remark~\ref{2025_05_10_23:55}.
It cuts each fibre $S_{(1:b)}$ (where $b\neq0$) at its only point $((0:1:0),(1:b))$ of order two in the abelian group $S_{(1:b)}$; see Proposition~\ref{2025_10_13_16:15}.

The two horizontal divisors $\pp$ and $\fq$ meet at the singular point $M \in S$, which lies in the fibre $Z$ over $(1:0)$.
The components $X$ and $Y$ of the fibre over $(0:1)$ intersect at the singular point $N \in S$, which belongs to the horizontal divisor $\pp$.

We will determine a desingularization of $S$ through a sequence of blowups, and after blowing down the fibre components that are contractible we will arrive at a minimal regular model of the fibration $\phi$,
which is uniquely determined by the function field $F|K=k(S)|k(\PP^1)$ according to a theorem of Lichtenbaum and Shafarevich (see \cite[Theorem~4.4]{Lic68}, \cite[p.\,155]{Sha66}, \cite[p.\,422]{Liu02}).

Let us describe the desingularization process explicitly. By blowing up the singular point $M$ we obtain an exceptional fibre consisting of two curves $A_1$ and $A_2$ that intersect at a unique point, which is the only singular point of the blowup surface. Blowing up this point produces an exceptional curve $\hat A_3$ consisting entirely of singular points. We blow up along the curve $\hat A_3$, which in turn gives us two exceptional curves $A_3$ and $A_4$, each covering the curve $\hat A_3$ isomorphically. 
This resolves the singular point $M$.

As for the point $N$, after blowing it up we get an exceptional curve $\hat B_1$ that contains only singular points. Blowing up along this curve produces two exceptional curves $B_1$ and $B_2$ that cover $\hat B_1$ isomorphically. The two exceptional curves intersect at two points $N_1$ and $N_2$, which are the only singular points on the blowup surface. These are rational double points of type $\boldsymbol{A_5}$ and $\boldsymbol{A_1}$, which are resolved by $3+1=4$ blowups, giving two bunches of exceptional curves $\{B_3,B_4,B_5,B_6,B_7\}$ and $\{B_8\}$ in the desingularized surface $\ov S$.


Therefore, we get in total four and eight exceptional curves lying over the singular points $M$ and $N$ respectively. 
By abuse of language,
let us denote the strict transforms of the curves $X$, $Y$, $Z$, $A_i$, $B_j$, $\pp$, $\fq$ by the same symbols.
The blowup computations show that the fibres of the composition morphism
\[ \bar f:\ov S \tto S \overset\phi\tto \PP^1 \]
over the points $(1:0)$ and $(0:1)$ are equal to
\begin{align*}
    \bar f^*(1:0) &= A_2 + 3 A_4 + 4 A_3 + 4 Z + 2 A_1, \\
    \bar f^*(0:1) &= X + 2 B_1 + 2 B_8 + 2 B_2 + 4 B_3 + 6 B_5 + 3 Y + 5 B_7 + 4 B_6 + 3 B_4.
\end{align*}
It also follows that the components of these fibres intersect as in Figure~\ref{2025_03_13_13:30}, where intersection numbers are equal to $1$ in all cases except $A_3 \cdot A_4=2$.
We see in particular that the arithmetic genus of the reducible reduced curves $A_1 \cup A_2 \cup A_3 \cup A_4$ and $B_1 \cup B_2 \cup \dots \cup B_8$ over the singular points $M$ and $N$ is equal to 1;
indeed, the arithmetic genus $p_a(C)$ of a connected reduced curve $C=\sum_i C_i$  on a smooth algebraic surface with canonical divisor $K$ is given by the Adjunction Formula 
(see e.g., \cite[4.1]{Lied13})
\[   2p_a(C) - 2 = C \cdot (C+K) = \sum_{i \neq j}C_i \cdot C_j + \sum_{i}C_i \cdot (C_i+K) = 2\sum_{i<j}C_i \cdot C_j + \sum_{i}(2 p_a (C_i) - 2). \]
The eight exceptional lines over $N$ form a cycle, which is one of the configurations associated to an elliptic singularity \cite[Proposition~7.6.9]{Ish18}.


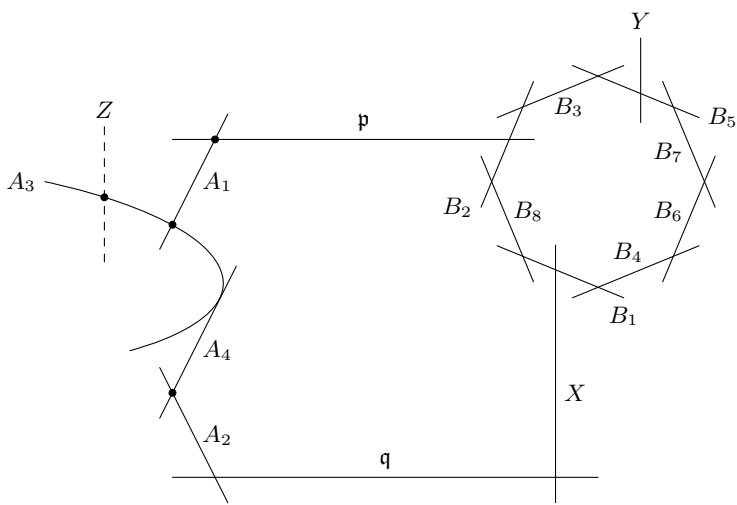
\begin{figure}[h]

\centering
\begin{tikzpicture}[line cap=round,line join=round,x=0.7cm,y=0.7cm, scale=0.8]

\begin{scriptsize}

\draw (0 -0.25*2.414213,-2.5 -0.25) -- (2.5/1.414213 +0.25*2.414213,-2.5/1.414213 +0.25);  \draw (.7,-2.1) node[above] {$B_4$};
\draw (2.5/1.414213 -0.25,-2.5/1.414213 -0.25*2.414213) -- (2.5 +0.25,0 +0.25*2.414213); \draw (2.2,-.7) node[left] {$B_6$};
\draw (2.5 +0.25,0 -0.25*2.414213) -- (2.5/1.414213 -0.25,2.5/1.414213 +0.25*2.414213); \draw (2.2,.7) node[left] {$B_7$};
\draw (2.5/1.414213 +0.25*2.414213,2.5/1.414213 -0.25) -- (0 -0.25*2.414213,2.5 +0.25); \draw (2.5/1.414213 +0.25*2.414213,2.5/1.414213 -0.25) node[right] {$B_5$};

\draw (0 +0.25*2.414213,2.5 +0.25) -- (-2.5/1.414213 -0.25*2.414213,2.5/1.414213 -0.25); \draw (-.7,2.2) node[below] {$B_3$};
\draw (-2.5/1.414213 +0.25,2.5/1.414213 +0.25*2.414213) -- (-2.5 -0.25,0 -0.25*2.414213); \draw (-2.5 -0.25,0 -0.25*2.414213) node[left] {$B_2$};
\draw (-2.5 -0.25,0 +0.25*2.414213) -- (-2.5/1.414213 +0.25,-2.5/1.414213 -0.25*2.414213); \draw (-2.2,-.7) node[right] {$B_8$};
\draw (-2.5/1.414213 -0.25*2.414213,-2.5/1.414213 +0.25) -- (0 +0.25*2.414213,-2.5 -0.25); \draw (0 +0.25*2.414213,-2.5 -0.25) node[below] {$B_1$};

\draw (-1.5,1) -- (-8 -1 -1,1); \draw (-5.5,1) node[above]{$\pp$};
\draw (-8 +.3 -1,1 +.6) -- (-9 -.3 -1,-1 -.6); \draw (-8.5 -1,0) node[right] {$A_1$};

\draw plot [smooth,tension=1.2] coordinates {(-9 -3 -1,1 -1) (-5 -3 +.09 -1,-1 -1) (-7 -3 -1,-3 -1)}; \draw (-9 -3 -1,1 -1) node[left] {$A_3$};

\draw (-8 +.5 -1,-3 +1) -- (-9 -.3 -1,-5 -.6); \draw (-8.5 -1,-4) node[right] {$A_4$};
\draw (-9 -.3 -1,-5 +.6) -- (-8 +.3 -1,-7 -.6); \draw (-8.5 -1,-6) node[right] {$A_2$};

\draw (-8 -1 -1,-7) -- (-1 +1,-7); \draw (-5,-7) node[above]{$\fq$};

\draw (-1,-1.5) -- (-1,-7 -.6); \draw (-1,-5) node[right] {$X$};

\draw (1,3.4) -- (1,1.4); \draw (1,3.4) node[above]{$Y$};
\draw[dashed] (-11.6,1 +.3) -- (-11.6,-1.7 -.3); \draw (-11.6,1 +.3) node[above]{$Z$};

\draw[fill=black] (-11.6,-.37) circle (1.6 pt);
\draw[fill=black] (-10,-1.02) circle (1.6 pt);
\draw[fill=black] (-8 -1,1) circle (1.6 pt);
\draw[fill=black] (-9 -1,-5) circle (1.6 pt);

\end{scriptsize}
\end{tikzpicture}
\caption{Configuration of curves on $\wt S$}\label{2025_03_13_13:30}
\end{figure}

Using the fact that fibres meet their components with intersection number zero we can compute the self-intersection numbers of all components. For $\bar f^*(1:0)$ we have
\[ A_1\cdot A_1 = -2, \quad A_2 \cdot A_2 = A_3\cdot A_3 = A_4 \cdot A_4 = -3, \quad Z \cdot Z = -1, \]
and for $\bar f^*(0:1)$ we get
\[ B_1\cdot B_1 = B_2\cdot B_2 = -3, \quad B_3 \cdot B_3 = \dots = B_8 \cdot B_8 = Y \cdot Y = X \cdot X = -2. \]
It follows that the fibration $\bar f$ is not minimal, in the sense that the fibre component 
$Z$ is a $(-1)$-curve, i.e., the rational curve $Z$ is contractible according to Castelnuovo's contractibility criterion.
Thus we can contract the $(-1)$-curve $Z$ to a smooth point,
and in turn get a smooth surface $\wt S$, a blowdown morphism $\ov S \to \wt S$, and a new fibration
\[ f: \wt S \tto \PP^1. \]
Note that the smooth surface $\wt S$ does not dominate $S$, i.e., the birational map $\wt S \ttto S$ is not a morphism, because no curve in $\wt S$ covers the curve $Z \su S$.
Note also that the horizontal curve $\pp \su \wt S$ is the moving singularity of the fibration $f$.

Let us summarize our discussion in the following theorem.

\begin{thm}
The fibration $f:\wt S \to \PP^1$ is the minimal regular model of the fibration $\phi:S\to \PP^1$, i.e., no fibre component of $f$ is a $(-1)$-curve. 
The fibres of $f$ over the points different from $(1:0)$ and $(0:1)$ coincide with those of $\phi$. The fibres over the points $(1:0)$ and $(0:1)$ are the Weil divisors
\begin{align*}
    f^*(1:0) &= A_2 + 3 A_4 + 4 A_3 + 2 A_1 \\
    f^*(0:1) &= X + 2 B_1 + 2 B_8 + 2 B_2 + 4 B_3 + 6 B_5 + 3 Y + 5 B_7 + 4 B_6 + 3 B_4,
\end{align*}
whose components intersect as in Figure~\ref{2025_03_13_13:30}, where intersection numbers are equal to $1$ in all cases except $A_3 \cdot A_4 = 2$.
These components have self-intersection numbers $X \cdot X = Y \cdot Y = -2$, 
$A_i \cdot A_i = -2$ if $i$ is odd, $A_i \cdot A_i = -3$ if $i$ is even, 
$B_j \cdot B_j = -3$ if $j \leq 2$, and $B_j \cdot B_j = -2$ if $j>2$.
\end{thm}


By construction, it is clear from Section~\ref{2025_10_01_15:20} that the fibration $f: \wt S \to \PP^1$ 
is an inseparable double cover of an elliptic fibration over $\PP^1$, defined as a closed subfibration of
$\eps: E \to \AAA^4$ over the hyperbola $\{ a=0, \, c=1, \, be = 1 \} \su \AAA^4_{(a,b,c,e)}$.
Let us give an explicit description.
Consider the projective surface $S'\su \PP^2 \times \PP^1$ cut out by
\begin{equation}\label{2025_03_13_15:55}
    s x z^2 + t (y^2 z + xyz + x^3) = 0,
\end{equation}
and the projection morphism
\[ \phi':S' \tto \PP^1, \]
which is a fibration by plane cubic curves. The fibre over a point $(1:b)$ with $b\neq0$ is an elliptic curve with $j$-invariant $j=b^{2}$, while 
the fibre over the point $(1:0)$ is a non-reduced reducible curve, with two components $X'$ and $Z'$ of multiplicities one and two that meet at the point
\[ M':= ((0:1:0),(1:0)). \]
The fibre over the point $(0:1)$ is a rational cubic curve $Y'$, with a nodal singularity at
\[ N':= ((0:0:1),(0:1)). \]
There are two horizontal prime divisors of degree 1 over the base $\PP^1$ that pass through the points $M'$ and $N'$, namely $\pp':= \{(0:1:0)\} \times \PP^1$ and $\fq':=\{(0:0:1)\} \times \PP^1$.


By the Jacobian criterion, the points $M'$ and $N'$ are the only singular points on the surface $S'$. They are Du Val singularities of type $\boldsymbol{A_6}$ and $\boldsymbol{A_1}$, which are resolved by $3+1=4$ blowups. 
Therefore, in the desingularized surface $\wt S'$ we get two bunches of exceptional curves $\{A_1',A_3',A_5',A_6',A_4',A_2' \}$ and $\{B_1' \}$, which intersect as in Figure~\ref{2025_03_13_15:15}, where we have denoted the strict transforms of the curves $X',Y',Z'$ and $\pp',\fq'$ with the same symbols.
Furthermore, the fibres $f'^*(1:0)$ and $f'^*(0:1)$ of the composition morphism
\[ f' : \wt S' \tto S' \overset {\phi'}\tto \PP^1 \]
are of type $\boldsymbol{\tilde E_7}$ and $\boldsymbol{\tilde A_1}$ respectively, or more precisely
\[ f'^*(1:0) = X' + 2A_2' + 3 A_4' + 4 A_6' + 2 Z' + 3 A_5' + 2 A_3' + A_1', \quad f'^*(0:1) = Y' + B_1'. \]
Their components have self-intersection numbers
$A_i'\cdot A_i' = B_1' \cdot B_1' = X' \cdot X' = Y' \cdot Y' = Z' \cdot Z' = -2$,
so in particular the elliptic fibration $f': \wt S' \to \PP^1$ is the minimal regular model of the fibration $\phi':S'\to \PP^1$.

\begin{figure}[h]

\centering
\begin{tikzpicture}[line cap=round,line join=round,x=0.7cm,y=0.7cm,scale=0.7]

\begin{scriptsize}

\draw plot [smooth,tension=1.3] coordinates {(-2,0 +1) (-.5,-5.5) (2,-3)}; \draw (2,-3) node[above] {$Y'$};
\draw plot [smooth,tension=1.3] coordinates {(-2,-10 -1) (-.5,-4.5) (2,-7)}; \draw (2,-7) node[below] {$B_1'$};

\draw (-2 +1,0) -- (-10 -1,0); \draw (-6,0) node[above] {$\pp'$};
\draw (-10  -0.40,0  + 0.40*2*0.70) -- (-10 +1.15  +0.30,0 -2*0.70*1.15  -0.30*2*0.70); \draw (-10  -0.40,0  + 0.40*2*0.70) node[above] {$A_1'$};
\draw (-10 +1.15  +0.30,0 -2*0.70*1.15  +0.30*2*0.70) -- (-10  -0.30,0 -4*0.70*1.15  -0.30*2*0.70); \draw (-10  +1.15/2 +.1,0 -3*0.70*1.15 +.3) node[left] {$A_3'$};
\draw (-10  -0.30,0 -4*0.70*1.15  +0.30*2*0.70) -- (-10 +1.15  +0.30,0 -6*0.70*1.15  -0.30*2*0.70); \draw (-10 +1.15/2 -.1,0 -5*0.70*1.15 +.2) node[right] {$A_5'$};
\draw (-10 +1.15  +0.40,0 -6*0.70*1.15  +0.40*2*0.29) -- (-10  -2.50,0 -6*0.70*1.15 -2*0.29*1.15  -2.50*2*0.29); \draw (-10  -2.50,0 -6*0.70*1.15 -2*0.29*1.15  -2.50*2*0.29) node[left] {$A_6'$};
\draw (-10  -0.30,0 -6*0.70*1.15 -2*0.29*1.15  +0.30*2*0.70) -- (-10 +1.15  +0.30,0 -6*0.70*1.15 -2*0.29*1.15  -2*0.70*1.15  -0.30*2*0.70); \draw (-10 +1.15/2 +.1,0 -7*0.70*1.15 -2*0.29*1.15 -.2) node[left] {$A_4'$};
\draw (-10 +1.15  +0.30,0 -6*0.70*1.15 -2*0.29*1.15  -2*0.70*1.15  +0.30*2*0.70) -- (-10  -0.30,0 -6*0.70*1.15 -2*0.29*1.15  -4*0.70*1.15  -0.30*2*0.70); \draw (-10 +1.15/2 -.2,0 -7*0.70*1.15 -2*0.29*1.15  -2*0.70*1.15 -.2) node[right] {$A_2'$};
\draw (-10,0 -6*0.70*1.15 -2*0.29*1.15  -4*0.70*1.15 +0.6) -- (-10,0 -6*0.70*1.15 -2*0.29*1.15  -6*0.70*1.15 -0.6); \draw (-10,0 -6*0.70*1.15 -2*0.29*1.15  -6*0.70*1.15 -0.6) node[below] {$X'$};

\draw (-10 -1,0 -6*0.70*1.10 -2*0.29*1.10  -6.5*0.70*1.10 ) -- (-2 +1,0 -6*0.70*1.10 -2*0.29*1.10  -6.5*0.70*1.10 ); \draw (-6,0 -6*0.70*1.10 -2*0.29*1.10  -6.5*0.70*1.10 ) node[above] {$\fq'$};

\draw (-10  -1.50,0 -6*0.70*1.15 -2*0.29*1.15  -1.50*2*0.29 -1) -- (-10  -1.50,0 -6*0.70*1.15 -2*0.29*1.15  -1.50*2*0.29 +1);  \draw (-10  -1.50,0 -6*0.70*1.15 -2*0.29*1.15  -1.50*2*0.29 +1) node[above]{$Z'$};

\end{scriptsize}
\end{tikzpicture}
\caption{Configuration of curves on $\wt S'$}\label{2025_03_13_15:15}
\end{figure}
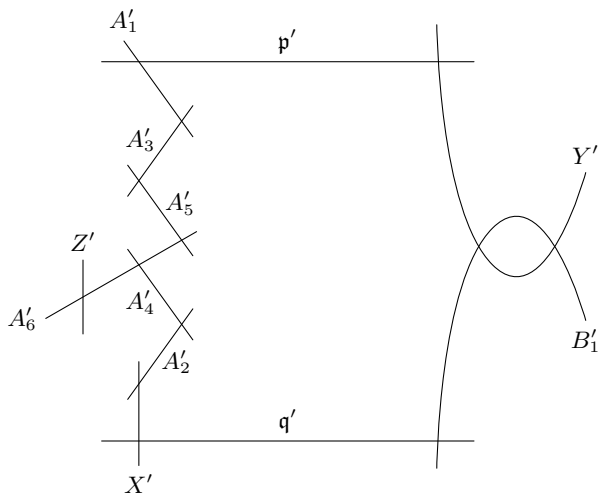

As follows from Section~\ref{2025_10_01_15:20} (see Proposition~\ref{2025_10_26_15:00}), the desired inseparable double cover of the elliptic fibration $f' : \wt S' \to \PP^1$ by our fibration $f : \wt S \to \PP^1$ is given by the rational map
\[ \wt S \ttto \wt S' \]
that is induced by the rational map 
\[ S \ttto S',  \quad ((x:y:z),(s:t)) \mapsto ((s x^2 : s x y : s z^2 + t y^2),(s:t)). \]
Note that $S \ttto S'$ is defined everywhere except at the two singular points $M,N \in S$.

As described in Section~\ref{2025_10_01_15:20}, over each point $(1:b)\neq (1:0)$ the fibre in $\wt S$ is mapped onto the fibre in $\wt S'$ by $\wt S \ttto \wt S'$ according to the rule $(x:y:z) \mapsto (x^2 : xy : z^2 + b y^2)$. 
Over the points $(1:0)$ and $(0:1)$ the following holds: 
the curves 
$A_1$, $A_4$, $A_5$, $A_2$, $B_1$, $B_2$ 
are taken isomorphically to the curves 
$A_3'$, $A_6'$, $A_4'$, $X'$, $B_1'$, $Y'$, respectively, 
and the three bunches of curves $\{B_8\}$, $\{ B_3,\dots,B_7,Y\}$ and $\{X\}$ are contracted to the three intersection points between the curves $B_1'$, $Y'$ and $\pp'$.
The horizontal divisor $\fq$ is mapped isomorphically to the horizontal divisor $\fq'$, and the moving singularity $\pp$ is mapped to the horizontal divisor $\pp'$ via an inseparable map of degree $2$.

We note that the curves $A_1'$, $A_5'$, $A_2'$, $Z'$ in Figure~\ref{2025_03_13_15:15} have not been covered by any curve in Figure~\ref{2025_03_13_13:30}.
In fact, the indeterminacy locus of the rational map $\wt S \ttto \wt S'$ consists of the four marked points in Figure~\ref{2025_03_13_13:30},
which, upon resolution, give rise to four exceptional curves that cover isomorphically the curves $A_1'$, $A_5'$, $A_2'$, $Z'$.
One of these exceptional curves is the curve $Z \su \ov S$ that was contracted by the map $\ov S \to \wt S$. 

Let us note finally that both surfaces $\wt S$ and $\wt S'$ are rational.
For $\wt S'$ this is clear because $\wt S' \to \PP^1$ is a cubic pencil (see \cite[Section~7.5]{SchSh19}). Precisely, the pencil is defined by equation~\eqref{2025_03_13_15:55}, and the surface $\wt S'$ is constructed from $\PP^2$ after resolving the two base points $(0:0:1)$ and $(0:1:0)$ of the pencil. 
Over $(0:0:1)$ we get the $(-2)$-curve $B_1'$ together with the section $\fq'$, which is a $(-1)$-curve. 
Over $(0:1:0)$ we get the $(-2)$-curves $A_1'$, $A_3'$, $A_5'$, $A_6'$, $A_4'$, $A_2'$ together with the setion $\pp'$, which is also a $(-1)$-curve.
The rational curves $X'$, $Y'$, $Z'$ are the strict transforms of the plane curves $\{x=0\}$, $\{y^2 z + xyz + x^3=0\}$, $\{ z=0 \}$ in $\PP^2$ respectively.

The rationality of $\wt S$ follows from the blow up computations along the curve $\hat A_3$. More precisely, there is a birational map $S \ttto \PP^2$ given by the assignment
\[ ((x:y:z),(s:t)) \mapsto (t x^2 y^3 : s x^5 : t y^4 z). \]
This rational map is undefined only at the two singular  points $M,N\in S$.




\begin{bibdiv}
\begin{biblist}
\bib{BedSt87}{article}{
  author={Bedoya, Hernando},
  author={St\"ohr, Karl-Otto},
  title={An algorithm to calculate discrete invariants of singular primes in function fields},
  journal={J. Number Theory},
  volume={27},
  date={1987},
  number={3},
  pages={310--323},
}

\bib{BM76}{article}{
  author={Bombieri, Enrico},
  author={Mumford, David},
  title={Enriques' classification of surfaces in char. $p$. III},
  journal={Invent. Math.},
  volume={35},
  date={1976},
  pages={197--232},
}

\bib{BMSa25}{article}{
  author={Borelli, Giuseppe},
  author={Dorado Moreira, Camilo David},
  author={Salom\~ao, Rodrigo},
  title={Non-smooth regular curves via a descent approach},
  date={2025},
  note={Preprint at \href {https://arxiv.org/abs/2501.17353}{\textsf {arXiv:2501.17353}}},
}

\bib{EGA}{article}{
  author={Grothendieck, Alexander},
  title={\'El\'ements de g\'eom\'etrie alg\'ebrique (r\'edig\'es avec la collaboration de Jean Dieudonn\'e)},
  journal={Inst. Hautes Études Sci. Publ. Math.},
  date={{1960--1967}},
  number={4, 8, 11, 17, 20, 24, 28, 32},
  label={EGA},
}

\bib{Har77}{book}{
  author={Hartshorne, Robin},
  title={Algebraic geometry},
  series={Graduate Texts in Mathematics},
  volume={52},
  publisher={Springer-Verlag, New York-Heidelberg},
  date={1977},
  pages={xvi+496 pp.},
}

\bib{HiSt22}{article}{
  author={Hilario, Cesar},
  author={St\"ohr, Karl-Otto},
  title={On regular but non-smooth integral curves},
  journal={J. Algebra},
  volume={661},
  date={2025},
  pages={278--300},
}

\bib{HiSt24}{article}{
  author={Hilario, Cesar},
  author={St\"ohr, Karl-Otto},
  title={Fibrations by plane projective rational quartic curves in characteristic two},
  date={2024},
  note={Preprint at \href {https://arxiv.org/abs/2409.05464}{\textsf {arXiv:2409.05464}}},
}

\bib{Ish18}{book}{
  author={Ishii, Shihoko},
  title={Introduction to singularities},
  edition={2},
  publisher={Springer, Tokyo},
  date={2018},
  pages={x+236},
}

\bib{Lic68}{article}{
  author={Lichtenbaum, Stephen},
  title={Curves over discrete valuation rings},
  journal={Amer. J. Math.},
  volume={90},
  date={1968},
  pages={380--405},
}

\bib{Lied13}{article}{
  author={Liedtke, Christian},
  title={Algebraic surfaces in positive characteristic},
  conference={ title={Birational geometry, rational curves, and arithmetic}, },
  book={ series={Simons Symp.}, publisher={Springer, Cham}, },
  date={2013},
  pages={229--292},
}

\bib{Liu02}{book}{
  author={Liu, Qing},
  title={Algebraic geometry and arithmetic curves},
  series={Oxford Graduate Texts in Mathematics},
  volume={6},
  publisher={Oxford University Press, Oxford},
  date={2002},
  pages={xvi+576 pp.},
}

\bib{Mat89}{book}{
  author={Matsumura, Hideyuki},
  title={Commutative ring theory},
  series={Cambridge Studies in Advanced Mathematics},
  volume={8},
  edition={2},
  note={Translated from the Japanese by M. Reid},
  publisher={Cambridge University Press, Cambridge},
  date={1989},
  pages={xiv+320},
}

\bib{Sal11}{article}{
  author={Salom\~ao, Rodrigo},
  title={Fibrations by nonsmooth genus three curves in characteristic three},
  journal={J. Pure Appl. Algebra},
  volume={215},
  date={2011},
  number={8},
  pages={1967--1979},
}

\bib{Sc09}{article}{
  author={Schr\"oer, Stefan},
  title={On genus change in algebraic curves over imperfect fields},
  journal={Proc. Amer. Math. Soc.},
  volume={137},
  date={2009},
  number={4},
  pages={1239-1243},
}

\bib{SchSh19}{book}{
  author={Sch\"utt, Matthias},
  author={Shioda, Tetsuji},
  title={Mordell-Weil lattices},
  series={Ergebnisse der Mathematik und ihrer Grenzgebiete. 3. Folge. A Series of Modern Surveys in Mathematics [Results in Mathematics and Related Areas. 3rd Series. A Series of Modern Surveys in Mathematics]},
  volume={70},
  publisher={Springer, Singapore},
  date={2019},
  pages={xvi+431},
}

\bib{Sha66}{book}{
  author={Shafarevich, Igor R.},
  title={Lectures on minimal models and birational transformations of two dimensional schemes},
  series={Tata Institute of Fundamental Research Lectures on Mathematics and Physics},
  volume={37},
  note={Notes by C. P. Ramanujam},
  publisher={Tata Institute of Fundamental Research, Bombay},
  date={1966},
  pages={iv+175 pp.},
}

\bib{St04}{article}{
  author={St\"ohr, Karl-Otto},
  title={On Bertini's theorem in characteristic $p$ for families of canonical curves in $\PP ^{(p-3)/2}$},
  journal={Proc. London Math. Soc. (3)},
  volume={89},
  date={2004},
  number={2},
  pages={291--316},
}

\bib{St07}{article}{
  author={St\"ohr, Karl-Otto},
  title={On Bertini's theorem for fibrations by plane projective quartic curves in characteristic five},
  journal={J. Algebra},
  volume={315},
  date={2007},
  number={2},
  pages={502--526},
}

\bib{Tate52}{article}{
  author={Tate, John},
  title={Genus change in inseparable extensions of function fields},
  journal={Proc. Amer. Math. Soc.},
  volume={3},
  date={1952},
  pages={400--406},
}

\bib{Tate74}{article}{
  author={Tate, John},
  title={The arithmetic of elliptic curves},
  journal={Invent. Math.},
  volume={23},
  date={1974},
  pages={179--206},
}
\end{biblist}
\end{bibdiv}
\end{document}